  \documentclass[12pt]{amsart}
\usepackage{amsmath}
\usepackage{amsfonts}
\usepackage{amscd}

\usepackage{stmaryrd}
\usepackage{amssymb}
\usepackage{graphicx}
\usepackage{caption}
\usepackage{subcaption}
\usepackage{dsfont}
\usepackage{color}
\usepackage{hyperref}
\usepackage{bbm}
\usepackage{a4wide}
\usepackage{sseq}
\usepackage{tikz-cd} 
\usepackage[abs]{overpic}

\usepackage{psfrag} 
\usepackage{marvosym}
\usepackage{amssymb}
\usepackage{amsthm}
\usepackage{mathrsfs}
\usepackage{enumitem}
\usepackage{bbm}
 \usepackage{comment}
\usepackage{ytableau}
\usepackage{hyperref}

\usepackage{calligra}
\usepackage{mathrsfs}
\DeclareMathAlphabet{\mathcalligra}{T1}{calligra}{m}{n}
\DeclareFontShape{T1}{calligra}{m}{n}{<->s*[1.5]callig15}{}

\newtheorem{theorem}{Theorem}[section]

\newtheorem{lemma}[theorem]{Lemma}
\newtheorem{proposition}[theorem]{Proposition}
\newtheorem{corollary}[theorem]{Corollary}

\theoremstyle{definition}
\newtheorem{definition}[theorem]{Definition}

\newtheorem{remark}[theorem]{Remark}

\newtheorem{theorem-definition}[theorem]{Theorem-Definition}

\numberwithin{equation}{section}

\newcommand{\CC} {\mathbb{C}}

\newcommand{\RR} {\mathbb{R}}

\newcommand{\ZZ} {\mathbb{Z}}

\newcommand {\shD} {\mathcal{D}}

\newcommand {\shH} {\mathcal{H}}
\newcommand {\shI} {\mathcal{I}}

\newcommand {\shL} {\mathcal{L}}

\newcommand {\shP} {\mathcal{P}}

\newcommand {\fot}  {\mathfrak{t}}

\newcommand{\sExt}{\mathscr{E} \kern -1pt xt}

\newcommand {\di} {\operatorname{div}}

\newcommand {\sHom}{\mathscr{H}\kern-5pt\mathcalligra{om}}

\renewcommand {\Im} {\operatorname{Im}}

\renewcommand {\ker } {\operatorname{Ker}}
\newcommand {\Ker} {\operatorname{Ker}}

\newcommand{\Diff}{\operatorname{Diff}}

\newcommand {\rank} {\operatorname{rank}}

\newcommand {\supp} {\operatorname{supp}}

\newcommand{\sTor}{\mathscr{T} \kern -3pt or}

\newcommand {\vol} {\operatorname{vol}}

\begin{document}
\title[GQ associated to mixed polarizations on K\"ahler manifolds with T-symmetry]{Geometric quantizations associated to mixed \\polarizations on K\"ahler manifolds with T-symmetry}	

\author[Leung]{Naichung Conan Leung,}
\address{The Institute of Mathematical Sciences and Department of Mathematics\\ The Chinese University of Hong Kong\\ Shatin\\ Hong Kong}
\email{leung@math.cuhk.edu.hk}
	
\author[Wang]{Dan Wang}
\address{The Institute of Mathematical Sciences and Department of Mathematics\\ The Chinese University of Hong Kong\\ Shatin\\ Hong Kong}
\email{dwang@math.cuhk.edu.hk}
\thanks{}

\maketitle
\begin{abstract}
Let $M$ be a compact K\"ahler manifold equipped with a pre-quantum line bundle $L$. In \cite{LW1}, using $T$-symmetry, we constructed a polarization $\mathcal{P}_{\mathrm{mix}}$ on $M$, which generalizes real polarizations on toric manifolds. In this paper, we obtain the following results for the quantum space $\mathcal{H}_{\mathrm{mix}}$ associated to $\mathcal{P}_{\mathrm{mix}}$. First, $\shH_{\mathrm{mix}}$ consists of distributional sections of $L$ with supports inside $\mu^{-1}(\mathfrak{t}^{*}_{\ZZ})$. 
This gives $\mathcal{H}_{\mathrm{mix}}=\bigoplus_{\lambda \in \mathfrak{t}^{*}_{\mathbb{Z}} } \mathcal{H}_{\mathrm{mix}, \lambda}$.  
Second, the above decomposition of $\shH_{\mathrm{mix}}$ coincides with the weight decomposition for the $T$-symmetry. Third, an isomorphism $\mathcal{H}_{\mathrm{mix}, \lambda} \cong H^{0}( M\sslash_{\lambda}T, L\sslash_{\lambda}T)$, for regular $\lambda$. Namely, geometric quantization commutes with symplectic reduction.
\end{abstract}

\section{Introduction}
    Let $(M, \omega)$ be a symplectic manifold equipped with a pre-quantum line bundle $(L, \nabla)$, in particular $F_{\nabla} = -i \omega$. A polarization $\shP$ on $M$ is an integrable Lagrangian subbundle of $TM\otimes \CC$.
Geometric quantization assigns a Hilbert space $\shH_{\shP}$ to these data.
      Namely,
    \begin{equation}
    \shH_{\shP} =  \Gamma(M, L) \cap \Ker( \nabla)|_{ \shP}, \end{equation}
where $\Gamma(M, L)$ is the space of smooth sections of $L$ \footnote{In fact we need to allow distributional sections.}. When $(M, \omega, J)$ is K\"ahler, we have a K\"ahler polarization $\shP_{J}= T^{0,1}_{J}M$, and $ \shH_{\shP_{J}}= H^{0}(M,L)$. If, in addition, $M$ admits $T$-symmetry, we constructed a polarization $\shP_{\mathrm{mix}}$ using $T$-symmetry in \cite{LW1}. In this paper, we study the quantum space $\shH_{\mathrm{mix}}$ associated to $\shP_{\mathrm{mix}}$ on $M$. Concretely, we have the following assumption throughout this paper.
\begin{enumerate}
\item[$(*)$] $(M, \omega, J)$ is a compact K\"ahler manifold of real dimension $2m$ equipped with an effective Hamiltonian $n$-dimensional torus action $\rho: T^{n} \rightarrow \Diff(M, \omega, J)$ by isometries with moment map $\mu: M \rightarrow \fot^{*}$. Let $(L, \nabla, h)$ be a $T^{n}$-invariant pre-quantum line bundle on $M$.
\end{enumerate}
Recall from \cite{LW1}, a (singular) polarization $\shP_{\mathrm{mix}}$ is constructed in this situation, which is, $$\shP_{\mathrm{mix}}=(\shP_{J} \cap \shD_{\CC}) \oplus \shI_{\CC},$$ where $\shD_{\CC}=(\ker d\mu) \otimes \CC$ and $\shI_{\CC}= (\Im d\rho)\otimes \CC$. When $n=m$, i.e. $M$ is a toric variety, $\shP_{\mathrm{mix}}$ coincides with the singular real polarization defined by moment map and $\shH_{\mathrm{mix}}$ is the space of Bohr-Sommerfeld states.

Recall that there is a natural way to embed the space of smooth sections into the space of distributional sections using the Liouville measure $\vol_{M}=\frac{\omega^{m}}{m!}$. That is, for any test section $\tau \in \Gamma_{c} (M, L^{-1})$,
 \begin{align*}
 \iota: \Gamma(M,L)  \rightarrow  \Gamma_{c} (M, L^{-1})' , ~~~s \mapsto( \iota s)(\tau) = \int_{M} \langle s, \tau \rangle \vol_{M}.
\end{align*}
Then the quantum space $\shH_{\mathrm{mix}}$ can be described as:
\begin{align*}
    \shH_{\mathrm{mix}} =  \Gamma_{c} (M, L^{-1})' \cap \Ker( \nabla)|_{ \shP_{\mathrm{mix}}}, \end{align*}

 Our first result says that $\shH_{\mathrm{mix}}$ consists of distributional sections with supports inside $\mu^{-1}(\fot^{*}_{\ZZ})$ and $\shH_{\mathrm{mix}, \lambda}$ is the $\lambda$-weight subspace of $\shH_{\mathrm{mix}}$.
 
 \begin{theorem}(Theorem \ref{thm3-0-1}) Under the assumption $(*)$,  
\begin{enumerate}
\item given any $\delta \in \shH_{\mathrm{mix}}$, we have $ \supp \delta \subset \bigcup _{\lambda \in \fot^{*}_{\ZZ} } \mu^{-1} (\lambda)$. This gives the following decomposition 
$$\shH_{\mathrm{mix}} = \bigoplus_{\lambda \in \fot_{\ZZ}^{*}} \shH_{\mathrm{mix}, \lambda},$$ 
where $\shH_{\mathrm{mix},\lambda}=\{\delta \in \shH_{\mathrm{mix}} \mid \mathrm{supp} ~\delta \subset \mu^{-1}(\lambda)\};$
\item for any  $\lambda \in \fot^{*}_{\ZZ}$, $\shH_{\mathrm{mix},\lambda}$ is a $\lambda$-weight subspace in $\shH_{\mathrm{mix}}$. 
\end{enumerate}
 Therefore the decomposition
$\shH_{\mathrm{mix}} = \bigoplus_{\lambda \in \fot_{\ZZ}^{*}} \shH_{\mathrm{mix}, \lambda}$ is the weight decomposition with respect to $T^{n}$-action.
\end{theorem}
 When $n=m$, this is a result for toric variety (see \cite{BFMN, Ham1}). Inspired by the works (see \cite{GS1}) of Guillemin and Sternberg that geometric quantizations commute with symplectic reductions, we give a geometric description of $\shH_{\mathrm{mix}, \lambda}$. Our main result (Theorem \ref{thm1-5} or Theorem \ref{thm4-0-11}) says that when $\lambda$ is an integral regular value of $\mu$, denoted as $\lambda \in \fot^{*}_{\ZZ, \mathrm{reg}}$, we have 
$$\shH_{\mathrm{mix},\lambda} \cong H^{0}(M_{\lambda}, L_{\lambda}),$$
 where $(M_{\lambda},L_{\lambda})=(M\sslash_{\lambda}T, L\sslash_{\lambda}T)$ is the symplectic reduction of $(M, L)$. 
Concretely, $M_{\lambda}=\mu^{-1}/T$, we also denote the level set $\mu^{-1}(\lambda)$ as $M^{\lambda}$. The restriction $(L^{\lambda}, \nabla)$ of pre-quantum line $(L, \nabla)$ to  $M^{\lambda}$ can be descended to the quotient space $M_{\lambda}$ denoted by $(L_{\lambda}, \nabla)$ (see \cite{GS1}).

Our second result states that, for any $s \in H^{0}(M_{\lambda}, L_{\lambda})$, there is an associated distributional section $\delta^{s} \in \Gamma_{c}(M^{\lambda}, (L^{\lambda})^{-1})'$ such that $\imath(\delta^{s})$ lies in $\shH_{\mathrm{mix}, \lambda}$, where $\imath:\Gamma_{c} (M^{\lambda}, (L^{\lambda})^{-1})' \hookrightarrow \Gamma_{c}(M, L^{-1})' $ is the natural inclusion.

 \begin{definition}(Definition \ref{def4-0-1})
For any  $\lambda \in \fot^{*}_{\ZZ, \mathrm{reg}}$ and $s \in H^{0}(M_{\lambda}, L_{\lambda})$, we define the {\em distributional section $\delta^{s} \in \Gamma_{c}(M^{\lambda},(L^{\lambda})^{-1})'$ associated to $s$} as follows: for any $\tau \in \Gamma_{c}(M^{\lambda},(L^{\lambda})^{-1})$,
$$\delta^{s} (\tau)= \int_{M^{\lambda}} \left\langle \pi^{*}s,\tau \right\rangle \vol^{\lambda},$$
where $\vol^{\lambda}$ is the volume form on $M^{\lambda}$ and $\pi: M^{\lambda} \rightarrow M_{\lambda}$ is the quotient map.
\end{definition}
 
 To see $\imath(\delta^{s}) \in \shH_{\mathrm{mix}, \lambda}$, we need to investigate the interaction between the covariant derivative on the space of smooth sections of $L$ and the covariant derivative on the space of distributional sections of $L^{\lambda}$. Our third result (Theorem \ref{thm3-5}) says that the following diagram
$$
\begin{tikzcd}
	\Gamma(M, L) \arrow{d} \arrow[hook]{r}{\vol_{M}}& \Gamma_{c}(M, L^{-1})'  \arrow{r}{\nabla_{\xi}}& \Gamma_{c}(M, L^{-1})'\\
	\Gamma(M^{\lambda}, L^{\lambda}) \arrow[hook]{r}{\vol^{\lambda}} &\Gamma_{c}(M^{\lambda}, (L^{\lambda})^{-1})' \arrow[hook]{u}{\imath} \arrow{r}{\nabla_{\xi}}& \Gamma_{c}(M^{\lambda}, (L^{\lambda})^{-1})' \arrow[hook]{u}{\imath},
\end{tikzcd}
$$ 
is a commutative diagram,
for any $\xi \in \Gamma(M, TM \otimes \CC)$ satisfying $\xi|_{M^{\lambda}} \in \Gamma(M^{\lambda}, TM^{\lambda} \otimes \CC)$.

In order to show the above diagram commutes, we use the coisotropic embedding theorem due to Weinstein \cite{We1} and further studied by Guillemin in \cite{Gui1} to relate the volume forms $\vol_{M}$ and $\vol^{\lambda}$. Then we obtain the following theorem:

\begin{theorem} (Theorem \ref{thm3-5}) For any $\lambda \in \fot^{*}_{\ZZ, \mathrm{reg}} $, $\delta \in \Gamma_{c}(M^{\lambda}, (L^{\lambda})^{-1})' $ and $\xi \in \Gamma(M, TM \otimes \CC)$ satisfying $\xi|_{M^{\lambda}} \in \Gamma(M^{\lambda}, TM^{\lambda} \otimes \CC)$, we have
\begin{equation}
\nabla_{\xi}(\imath(\delta))=\imath(\nabla_{\xi}\delta),
\end{equation}
where $\imath:\Gamma_{c} (M^{\lambda}, (L^{\lambda})^{-1})' \hookrightarrow \Gamma_{c}(M, L^{-1})' $ is the natural inclusion.
\end{theorem} 

\begin{proposition}\label{thm1-3}(Proposition \ref{pro4-0-4})
For any $\lambda \in \fot^{*}_{\ZZ, \mathrm{reg}} $ and $s \in H^{0}(M_{\lambda}, L_{\lambda})$, we have
$$\imath(\delta^{s}) \in \shH_{\mathrm{mix},\lambda},$$
where $\imath:\Gamma_{c} (M^{\lambda}, (L^{\lambda})^{-1})' \hookrightarrow \Gamma_{c}(M, L^{-1})' $ is the natural inclusion.
\end{proposition}

This allows us to define a map $\kappa: H^{0}(M_{\lambda}, L_{\lambda}) \rightarrow \shH_{\mathrm{mix},\lambda}$ by $s \mapsto \kappa(s)=\imath(\delta^{s})$. Finally we show that $\kappa$ is an isomorphism.
\begin{theorem} \label{thm1-5}(Theorem \ref{thm4-0-11})
 For any $\lambda \in \fot^{*}_{\ZZ, \mathrm{reg}}$, $$ \kappa: H^{0}(M_{\lambda}, L_{\lambda}) \rightarrow \shH_{\mathrm{mix},\lambda}$$ is an isomorphism.
\end{theorem}

We show the surjectivity of $\kappa$ via the following steps.
  First, we show that any element $\tilde{\delta}$ in $\shH_{\mathrm{mix}, \lambda}$ is locally a delta function along $\mu^{-1}(\lambda)$, and does not involve any derivative of delta functions. This implies $\tilde{\delta} = \imath(\delta)$, for some $\delta \in \Gamma_{c}(M^{\lambda}, (L^{\lambda})^{-1})'$.
  
  Second, we show that $T^{n}$-invariant distributional sections of $L^{\lambda}$ can be descended to distributional sections of $L_{\lambda}$. That is, for any $\delta \in \Gamma_{c} (M^{\lambda}, (L^{\lambda})^{-1})'$ satisfying $\nabla_{\xi^{\#}} \delta=0$, there exists a distributional section $\eta \in \Gamma_{c}(M_{\lambda}, L_{\lambda}^{-1})'$ such that $\delta= \pi^{*} \eta$ (Lemma \ref{lem4-0-6}).
  
 Third, we show that if $\nabla_{\zeta}( \pi^{*} \eta)=0$ for all $\zeta \in \Gamma(M^{\lambda}, \shP_{\mathrm{mix}})$, then $\eta$ is $\bar{\partial}$-closed (Theorem \ref{thm4-0-8}). By the regularity of elliptic operator $\Delta= \bar{\partial}^{*} \bar{\partial}$, we have $\eta$ is smooth (i.e. $\eta=\iota(s)$, for some $s \in H^{0}(M_{\lambda}, L_{\lambda})$). 
  Finally, we show that $\tilde{\delta}= \kappa(s)$.


\subsection{Acknowledgements}
We are grateful to Siye Wu for insightful comments and useful discussions. D. Wang would like to thank Qingyuan Jiang, Yutung Yau and Ki Fung Chan for many helpful discussions. This research was substantially supported by grants from the Research Grants Council of the Hong Kong Special Administrative Region, China (Project No. CUHK14301619 and CUHK14301721) and a direct grant from the Chinese University of Hong Kong.

\section{Preliminary}
\subsection{The Marsden-Weistein construction}
In this subsection, we review the basic concepts of Hamiltonian action and symplectic reduction in order to fix the notations in our setting (for more details, the reader can refer to \cite{GS1}, \cite{MW}).
\subsubsection{Hamiltonian action}
 Let $(M ,\omega)$ be a compact symplectic manifold. For $f\in C^{\infty }(M,\RR)$, the Hamiltonian vector field $X_{f}$ associated to $f$ is determined by 
$\imath _{X_{f}}\omega =-df$. This gives a Lie algebra homeomorphism $$\psi: (C^{\infty}(M;\RR), \{\cdot,\cdot \}) \rightarrow (\mathrm{Vect}(M, \omega),[\cdot , \cdot])$$ defined by $\psi(f) = X_{f}$, where $\{ ,\}$ is the Poisson bracket of two functions $f, g \in C^{\infty}(M;\RR)$ determined by $\{f, g\} = \omega(X_{f}, X_{g}).$
Let $T^{n}$ be a torus of real dimension $n$
and $\rho :T^{n}\rightarrow \mathrm{Diff}(M,\omega )$ an action of $T^{n}$
on $M$ which preserves $\omega $. Let $\fot$ be the Lie algebra of $T^{n}$. Differentiating $\rho $ at the identity
element, we have 
\begin{equation*}
d\rho :\mathfrak{t}\rightarrow \mathrm{Vect}(M,\omega ),~~~~\xi \mapsto \xi
^{\#},
\end{equation*}
where  $\mathfrak{t}$ is the Lie algebra of $T^{n}$ and $\xi ^{\#}$ is
called the fundamental vector field associated to $\xi $. 
The action of $T^{n}$ on $M$ is said to be {\em Hamiltonian} if $d\rho$ factors through $\psi$.

Let $\langle, \rangle: \fot^{*} \times \fot \rightarrow \RR$ be the natural pairing between $\fot^{*}$ and $\fot$. For each point $p \in M$, we can associate an element $\mu(p) \in \fot^{*}$ by the formula
$$\langle \mu(p),\xi \rangle=-\mu^{\xi}(p),~~ \forall \xi \in \fot.$$ 
This gives us a {\em moment mapping} $\mu: M \rightarrow \fot^{*}$
which is a $T^{n}$-equivariant map.

\subsubsection{Symplectic reduction}
We denote the set of regular values of $\mu$ by $\fot^{*}_{\mathrm{reg}}$, that is,
$$\fot^{*}_{\mathrm{reg}}= \{\lambda \in \fot^{*} | ~~\lambda \text{ is a regular value of} ~~\mu \}.$$

For any $\lambda \in \fot^{*}_{\mathrm{reg}}$, denote the level set $\mu^{-1}(\lambda)$ by $M^{\lambda}$. Then $M^{\lambda}$ is a $T^{n}$-invariant coisotropic submanifold ($i: M^{\lambda}\hookrightarrow M$) and the action of $T^{n}$ is locally free (see \cite{MW}). For simplicity, we assume $T^{n}$ acts freely on $M^{\lambda}$. Then the projection mapping
$$\pi: M^{\lambda} \rightarrow M_{\lambda}$$ is a principal $T^{n}$-fibration.
 Moreover there exists a unique symplectic form $\omega_{\lambda}$ on $M_{\lambda}$ such that $\pi^{*}\omega_{\lambda}=i^{*} \omega$.
Denote the volume form $\frac{1}{(m-n)!} \omega_{\lambda}^{m-n}$ on $M_{\lambda}$ by $\vol_{\lambda}$. Take a connection $\alpha \in \Omega^{1}(M^{\lambda}, \fot)$ on $M^{\lambda}$, $\pi^{*}\vol_{\lambda}\wedge \alpha^{n}$ is a volume form on $M^{\lambda}$ denoted by $\vol^{\lambda}$ (here $\alpha^{n}$ is a $n$-form on $M^{\lambda}$ defined by $\frac{\alpha \wedge \cdots \wedge \alpha}{v}$, where $v \in \wedge^{n} \fot$ is a $T^{n}$-invariant top form on $\fot$).

$$ 
\begin{tikzcd}
(M^{\lambda}=\mu^{-1}(\lambda), \vol^{\lambda}) \arrow[hook]{r}{i}  \arrow{d}{\pi} &  M  \\
(M_{\lambda}=\mu^{-1}(\lambda)/T^{n}, \vol_{\lambda}).
\end{tikzcd}
$$

 \subsection{Pre-quantum data}
 In this subsection, we first review the definition of $T^{n}$-invariant pre-quantum line bundles. Then we restate the result that the pre-quantum line bundle can always be descended to the reduction space by Guillemin and Sternberg in our setting (K\"ahler manifold equipped with $T^{n}$-symmetry). 
 
\begin{definition}
Let $(M, \omega,J )$ be a symplectic manifold, a {\em pre-quantum line bundle $(L,\nabla, h)$} on $M$ is a complex line bundle $L$ together with a Hermitian metric $h$ and Hermitian connection $\nabla$, such that the curvature form $F_{\nabla} = -i \omega$.  
\end{definition}
The existence of a pre-quantum line bundle $L$ on $M$ is equivalent to $[\frac{\omega}{2\pi}]$ being integral (see \cite{Kos}). 
When $M$ is K\"ahler, $L$ is an ample holomorphic line bundle.
There is a canonical representation of the Lie algebra $\fot$ on space of smooth sections of $L$ given by the operators 
\begin{equation}\label{eq2-0-1} \nabla_{\xi^{\#}} +i \mu^{\xi}, \xi \in \fot. \end{equation}
 {\em The pre-quantum line bundle is said to be $T^{n}$-invariant} if there exists a global action of $T^{n}$ on $L$ such that the induced action of $\fot$ is given by (\ref{eq2-0-1}). It is always possible if the $T^{n}$-action on $M$ is Hamiltonian (see \cite{Kos}).

Let $\fot_{\ZZ}$ be the kernel of the exponential map $\operatorname{exp}: T^{n} \rightarrow \fot$ and $\fot_{\ZZ}^{*} \subset \fot^{*}$ be the dual lattice of $\fot_{\ZZ}$. We denote the set of integral regular values of $\mu$ by $\fot^{*}_{\ZZ, \mathrm{reg}}$, that is, $\fot^{*}_{\ZZ, \mathrm{reg}}=\fot^{*}_{\mathrm{reg}} \cap \fot^{*}_{\ZZ}$.

Guillemin and Sternberg in \cite{GS1} showed that there are associated pre-quantum data on the reduction space $M_{\lambda}$, for $\lambda \in \fot^{*}_{\ZZ, \mathrm{reg}}$.
\begin{theorem}\label{thm2-0-1}\cite[Theorem 3.2]{GS1} There is a unique line bundle with connection  $( L_{\lambda}, \nabla_{\lambda})$ on $M_{\lambda}$ such that 
\begin{equation}
\pi^{*}{L_{\lambda}}=i^{*}L=: L^{\lambda}, ~~~ \text{and}~~\/~~  \pi^{*}\nabla_{\lambda}=i^{*}\nabla.
\end{equation}\end{theorem} 

\begin{corollary}\cite[Corollary 3.4]{GS1} The curvature of the connection, $\nabla_{\lambda}$, is the symplectic form $\omega_{\lambda}$.
\end{corollary}

Therefore we have the following commuting diagram:
$$
\begin{tikzcd}
&(L^{\lambda}, i^{*}\nabla)\arrow{rd}\arrow[hook]{rr}\arrow{ld}&&(L, \nabla)\arrow{rd}&&\fot^{*}_{\ZZ, \mathrm{reg}} \arrow[hook]{d}
\\
(L_{\lambda},\nabla_{\lambda})\arrow{rd}&&(M^{\lambda}, \vol^{\lambda})\arrow{ld}[swap]{\pi}\arrow[hook]{rr}{i}&&(M, \vol_{M}) \arrow{r}{\mu}&\fot^{*}\\
&(M_{\lambda}, \vol_{\lambda})&&&&\end{tikzcd}
$$
By abuse of notations, we denote both $i^{*}\nabla$ and $\nabla_{\lambda}$ by $\nabla$.
In order to pull-back distributional sections from $M_{\lambda}$ to $M^{\lambda}$ later, 
we first recall how to push-forward sections of line bundle $L^{\lambda}$. 
\begin{remark}\label{rk4-1}
Let $\pi: P \rightarrow B$ be a principal $T^{n}$-bundle over $B$, $E \rightarrow B$ a line bundle over $B$, and $\pi^{*}E \rightarrow P$ the pullback line bundle. 
  Then we can define the dual map
  $$\pi^{*}: \Gamma_{c}(B,L^{-1})' \rightarrow \Gamma_{c}(P, (\pi^{*}L)^{-1})', \eta \mapsto \pi^{*}\eta $$
by $(\pi^{*}\eta)(\tau)=\eta(\pi_{*}\tau)$, for any $\tau \in \Gamma_{c}(P, L^{-1})$
   \end{remark}
 Throughout this paper, we fix a $T^{n}$-invariant $n$-form $d\theta$ on $T^{n}$ such that $\int_{T^{n}}d\theta=1$.
  When we deal with the pull-back of distribution sections of $L_{\lambda}$, we mean in the sense of Remark \ref{rk4-1} with respect to $d\theta$.

\subsection{Complex structures on symplectic reduction spaces}
In order to study the relationship between geometric quantization associated to $\shP_{\mathrm{mix}}$ and symplectic reduction, we recall the work on the existence of complex structures on symplectic reduction spaces $M_{\lambda}$ (see \cite{GS1}).

Recall that the anti-holomorphic Lagrangian sub-bundle $TM^{0,1}_{J} \subset TM \otimes \CC$ is a K\"ahler polarization denoted by $\shP_{J}$. We define $F \subset TM^{\lambda} \otimes \CC$ by 
\begin{equation}F_{p}=(\shP_{J})_{p} \cap (TM^{\lambda} \otimes \CC)_{p},
\end{equation}
for any $p\in M^{\lambda}$. $F$ can be descended to a bundle $\shP_{J,\lambda}$ over the reduction space $M_{\lambda}$, which is a positive-definite Lagrangian sub-bundle of $TM_{\lambda} \otimes \CC$. Under the assumption $(*)$, we have $(\shD_{\CC} \cap \shP_{J})_{p}= F_{p}$, for any $p \in  M^{\lambda}$.

\begin{theorem}\cite[Theorem 3.5]{GS1} There is a positive-definite polarization $\shP_{J,\lambda}$ canonically associated with $\shP_{J}$ on the reduction space $M_{\lambda}$.
\end{theorem}

By  Definition 4.2 and Lemma 4.3 in \cite{GS1}, $\shP_{J,\lambda}$ determined a complex structure $J_{\lambda}$ on $M_{\lambda}$ such that 
\begin{equation}
\shP_{J,\lambda}= TM_{\lambda}^{0,1},
\end{equation}
where $TM_{\lambda}^{0,1}$ is the anti-holomorphic sub-bundle of $TM_{\lambda}\otimes \CC$ with respect to $J_{\lambda}$.

\subsection{ Polarizations on K\"ahler manifolds with $T$-symmetry}
In \cite{LW1}, we constructed polarizations $\shP_{\mathrm{mix}}$ on K\"ahler manifolds with $T$-symmetry. Throughout this sections, the existence of pre-quantum line $L$ in $(*)$ is not needed.

For any point $p \in M$, consider the map $\rho_{p}: T^{n} \rightarrow M$ defined by $\rho_{p}(g)= \rho(g)(p)$. Let $\shI_{\RR} \subset TM$ be the singular distribution generated by fundamental vector fields in $\Im d\rho$, that is $(\shI_{\RR})_{p} = \Im d\rho_{p}(e)$. Let $\shD_{\RR} = (\ker d\mu) \subset TM$ be  a distribution defined by the kernel of $d\mu$.  Note that there is a K\"ahler polarization $\shP_{J}=TM^{0,1}_{J}$ associated to the complex structure $J$ on a K\"ahler manifold $(M, \omega, J)$. 
   
\begin{definition}\label{def3-0-1}\cite[Definition]{LW1}
 We define {\em the singular distribution $\shP_{\mathrm{mix}} \subset TM \otimes \CC$} by:
\begin{equation}
\shP_{\mathrm{mix}} = (\shP_{J} \cap\shD_{\CC} ) \oplus \shI_{\CC},
\end{equation}
where $\shD_{\CC} = \shD_{\RR} \otimes \CC$ and $\shI_{\CC} = \shI_{\RR} \otimes \CC$ are the complexification of $\shD_{\RR}$ and $\shI_{\RR}$ respectively.
\end{definition}
 Let $H_{p}$ be the stabilizer of $T^{n}$ at point $p \in M$. Denote by $\check{M}$ the union of $n$-dimensional orbits in $M$, that is,
$$\check{M}= \{p \in M| \dim H_{p} =0\},$$
which is an open dense subset in $M$.
 \begin{theorem}\label{thm2-7}\cite[Theorem 1.1]{LW1}
Under the assumption $(*)$, $\shP_{\mathrm{mix}}$ is a singular polarization and smooth on $\check{M}$. Moreover, $\rank(\shP_{\mathrm{mix}} \cap \bar{\shP}_{mix}  \cap TM)|_{\check{M}}=n$.
\end{theorem}
According to Definition \ref{def4-7}, $\shP_{\mathrm{mix}}$ is a singular real polarization on $M$, when $n=m$, namely, $M$ is toric manifold; $\shP_{\mathrm{mix}}$ is a singular mixed polarization on $M$, when $1 \le n< m$. 
   
 \section{Main results}  
We define the quantum space associated to the polarization $\shP_{\mathrm{mix}}= (\shP_{J} \cap \shD_{\CC}) \oplus \shI_{\CC}$ as follows. 
Let $(L, \nabla, h)$ be the pre-quantum line bundle on $M$. We first recall the definition of quantum space $\shH$ associated to polarization $\shP$ (see \cite{Wo}).
\begin{definition}  {\em The quantum space $\shH$} associated to polarization $\shP$ is the following subspace of $\Gamma_{c}(M,L^{-1})'$: 
$$\shH = \{ \delta \in\Gamma_{c} (M,L^{-1})'\mid  \nabla _{\xi}\delta =0, ~\forall~ \xi \in \Gamma(M,\shP)\},$$
where $\nabla_{\xi}$ is the covariant derivative operator acting on the space of distributional sections defined by equation (\ref{eqcon}).

In our setting, even through the polarization $\shP_{\mathrm{mix}}$ is singular, we continue to use the above definition for the quantum space. We denote it by $\shH_{\mathrm{mix}}$. When $n=m$, $M$ is toric variety and $\shP_{\mathrm{mix}}$ is a singular real polarization. The definition of $\shH_{\mathrm{mix}}$ coincides with the definition of the quantum spaces associated to singular real polarizations studied in \cite{BFMN}.  Moreover, for any $\lambda \in \fot^{*}$, we define {\em the subspace of those sections with supports on $\mu^{-1}(\lambda)$} as:
$$\shH_{\mathrm{mix},\lambda}=\{\delta \in \shH_{\mathrm{mix}} \mid \mathrm{supp} ~\delta \subset \mu^{-1}(\lambda)\}.$$
\end{definition}

 \subsection{Distributional sections in $\shH_{\mathrm{mix}, \lambda}$ associated to sections in $H^{0}(M_{\lambda}, L_{\lambda})$}

In this subsection, we first confirm that for any distributional section $\delta \in \shH_{\mathrm{mix}}$ we have (see Theorem \ref{thm3-0-1}):
$$\supp \delta \subset \bigcup _{\lambda \in \fot^{*}_{\ZZ}} \mu^{-1} (\lambda).$$

After extending the $T^{n}$-action from the space of smooth sections to the space of distributional sections of $L$, we show that $\shH_{\mathrm{mix}, \lambda}$ is a $\lambda$-weight subspace of $\shH_{\mathrm{mix}}$, for any $\lambda \in \fot^{*}_{\ZZ}$. This gives the weight decomposition of $\shH_{\mathrm{mix}}$, i.e. $\shH_{\mathrm{mix}}=\bigoplus_{\lambda \in \fot^{*}_{\ZZ} } \shH_{\mathrm{mix}, \lambda}.$

Inspired by the work on geometric quantizations commute with symplectic reductions by Guillemin and Sternberg in \cite{GS1}, we expect to establish the isomorphism between $H^{0}(M_{\lambda}, L_{\lambda})$ and $ \shH_{\mathrm{mix},\lambda}$, where $(M_{\lambda},L_{\lambda})$ is the symplectic reduction of $M$ at a regular integral level $\lambda$. 

At the end of this subsection, given any holomorphic section $s \in H^{0}(M_{\lambda}, L_{\lambda})$, we define an associated distributional section $\delta^{s} \in \Gamma_{c}(M^{\lambda},(L^{\lambda})^{-1})'$ (see Definition \ref{def4-0-1}) with respect to the volume form $\vol^{\lambda}$.
Then we show that distributional sections in $\shH_{\mathrm{mix}, \lambda}$ associated to sections in $H^{0}(M_{\lambda}, L_{\lambda})$ (see Proposition \ref{pro4-0-4}). That is $\imath(\delta^{s}) \in \shH_{\mathrm{mix}, \lambda}$, where $\imath:\Gamma_{c} (M^{\lambda}, (L^{\lambda})^{-1})' \hookrightarrow \Gamma_{c}(M, L^{-1})' $ is the natural inclusion.

 In order to study the quantum space $\shH_{\mathrm{mix}}$, we recall how to extend covariant differentiation to distributional sections of $L$ (see  \cite{BFMN}). First of all, there is a natural way to embed the space of smooth sections into the space of distributional sections using the Liouville measure $\vol_{M}=\frac{\omega^{m}}{m!}$: 
 \begin{align*}
 \iota: \Gamma(M,L) & \rightarrow  \Gamma_{c} (M, L^{-1})' \\
 s                            & \mapsto( \iota s)(\tau) = \int_{M} \langle s, \tau \rangle \vol_{M}.
\end{align*}
Here $\langle, \rangle: L \times L^{-1} \rightarrow \CC$ is the natural paring between $L$ and $L^{-1}$. Let $\nabla$ be the connection on $L^{-1}$ such that $d\langle s ,\tau \rangle=\langle \nabla s, \tau\rangle+\langle s, \nabla \tau\rangle$. It is necessary to require that the operator $\nabla$ acting on the distributional sections $\iota(s)$ which come from any smooth section $s$ coincides with the operator $\nabla$ acting on $s$, i.e. the following diagram 
$$
\begin{tikzcd}
\Gamma(M,L) \arrow{r} {\iota} \arrow{d}{\nabla_{\xi}} &\Gamma_{c} (M,L^{-1})'\arrow{d}{\nabla_{\xi}} \\
\Gamma(M,L)\arrow{r}{\iota} &  \Gamma_{c} (M,L^{-1})'
\end{tikzcd}
$$
commutes, for any $\xi \in \Gamma(M, TM \otimes \CC)$.
Let $\di \xi$ be the divergence of $\xi$ with respect to $\vol_{M}=\frac{\omega^{m}}{m!}$, equivalently,
$\shL_{\xi}(\vol_{M})= (\di \xi )\vol_{M}$. It can be seen that:
\begin{align*}
0=\int_{M} \shL_{\xi}(\langle s,\tau \rangle\vol_{M}) &= \int_{M} (\shL_{\xi}\langle s,\tau \rangle)\vol_{M} +\int_{M} \langle s,\tau \rangle \shL_{\xi}(\vol_{M})\\                                                                           
 &= \int_{M} \langle \nabla_{\xi}s,\tau \rangle\vol_{M}+\int_{M}\langle s,\nabla_{\xi}\tau \rangle\vol_{M} +\int_{M}\langle s,\tau \rangle (\di \xi)\vol_{M}
                                                                             . \end{align*}
  This gives, for any smooth section $s \in \Gamma(M,L)$ and smooth test section $\tau \in \Gamma_{c}(M,L^{-1})$,
\begin{equation}\label{eqsec} (\nabla_{\xi}\iota(s))(\tau) = \int_{M} \langle \nabla_{\xi}s, \tau \rangle \vol_{M} =\int_{M} \langle s,-((\di\xi) \tau + \nabla \tau) \rangle\vol_{M}.\end{equation}
To determine $\nabla _{\xi} \sigma$ for a general distributional section $\sigma$ not of the form $\iota( s)$, we built its transpose by integrating the operator $\nabla_{\xi}$ by parts.
Namely, $\nabla_{\xi}$ is characterized by its transpose $^{t}\nabla_{\xi}$ as follows: for any $\tau \in \Gamma_{c}(M,L^{-1})$,
\begin{equation}\label{eqcon} (\nabla _{\xi} \sigma)(\tau) = \sigma(^{t}\nabla_{\xi} \tau), ~~~~\text{with}~~~~ ^{t}\nabla_{\xi} \tau = - (\di \xi \tau + \nabla_{\xi} \tau).\end{equation}

Similarly we can extend the $T^{n}$-action on space of smooth sections to space of distributional sections such that the inclusion $\iota: \Gamma(M, L) \hookrightarrow \Gamma_{c}(M, L^{-1})'$ (with respect to the Liouville volume form ${\vol_{M}}$) is $T^{n}$-equivariant. That is, for any $\xi \in \fot$, the following diagram commute.
$$
\begin{tikzcd}
	\Gamma(M, L) \arrow{d}{\xi \cdot} \arrow[hook]{r}{\vol_{M}}& \Gamma_{c}(M, L^{-1})' \arrow{d}{\xi \cdot}\\
		\Gamma(M, L)\arrow[hook]{r}{\vol_{M}} & \Gamma_{c}(M, L^{-1})'
\end{tikzcd}, ~~i.e. ~~~\xi \cdot (\iota(s))= \iota(\xi \cdot s).
$$
Namely, for any $\delta \in \Gamma_{c}(M,L^{-1})'$, $\tau \in \Gamma_{c}(M,L^{-1})$, and  $\xi \in \fot$, $\xi \cdot \delta$ is characterized by: 
\begin{equation}(\xi \cdot \delta ) (\tau)= \delta(\xi \cdot \tau),
\text{with}~~~~ \xi \cdot \tau= \nabla_{\xi^{\#}} \tau + i\mu^{\xi} \tau.
\end{equation}
The $T^{n}$-action on $L$ preserve connection $\nabla$, which implies that $T^{n}$ acts on $\shH_{\mathrm{mix}}$. We obtain the following results.

\begin{theorem}\label{thm3-0-1} Under the assumption $(*)$,    
\begin{enumerate}
\item given any $\delta \in \shH_{\mathrm{mix}}$, we have $ \supp \delta \subset \bigcup _{\lambda \in \fot^{*}_{\ZZ} } \mu^{-1} (\lambda)$. This gives the following decomposition 
$$\shH_{\mathrm{mix}} = \bigoplus_{\lambda \in \fot_{\ZZ}^{*}} \shH_{\mathrm{mix}, \lambda},$$ 
where $\shH_{\mathrm{mix},\lambda}=\{\delta \in \shH_{\mathrm{mix}} \mid \mathrm{supp} ~\delta \subset \mu^{-1}(\lambda)\};$
\item for any  $\lambda \in \fot^{*}_{\ZZ}$, $\shH_{\mathrm{mix},\lambda}$ is a $\lambda$-weight subspace in $\shH_{\mathrm{mix}}$. 
\end{enumerate}
 Therefore the decomposition
$\shH_{\mathrm{mix}} = \bigoplus_{\lambda \in \fot_{\ZZ}^{*}} \shH_{\mathrm{mix}, \lambda}$ is the weight decomposition with respect to $T^{n}$-action.
\end{theorem}

\begin{proof} 
(1) For a loop $\gamma_{b} \subset T^{n}$ specified by a vector $b\in \fot_{\ZZ}$, for any test function $\tau \in \Gamma_{c}(M,L^{-1}) $, parallel transporting $\tau(p)$ with respect to the connection $\nabla$ around a loop $\gamma_{b} \cdot p \subset M$ results in multiplication of $\tau(p)$ by $e^{-2i\pi \langle \mu(p), b \rangle}$, where $ \langle ,  \rangle: \fot^{*} \times \fot \rightarrow \RR$ is the natural pairing between $\fot^{*}$ and $\fot$. 
The reason is as follows.
Recall given $T^{1}$-equivariant line bundle $L \rightarrow M$ with equivariant curvature $F_{A}+ \mu$, the holonomy around any $T^{1}$-orbit at $p \in M$ is given by $e^{2\pi i\mu(p)}$. 
Applying this to our case, for a loop $\gamma_{b} \subset T^{n}$ specified by $b\in \fot_{\ZZ}$, holonomies of $(L^{-1}, \nabla)$ around the loops in $M$ specified by $b \in \fot_{\ZZ}$ define a smooth function: 
\begin{align*}
f_{b}: M \rightarrow \CC,  p\mapsto f_{b}(p):=e^{-2i\pi \langle \mu(p), b \rangle}.
\end{align*} 
Therefore, $\nabla_{b^{\#}} \tau=0$ implies $f_{b}\cdot \tau=\tau$. By transporting this to the dual space, we have $f_{b} \cdot \delta =\delta$ for any $\delta \in \Gamma_{c}(M,L^{-1})'$ satisfying $\nabla_{b^{\#}} \delta=0$.
For $\delta \in \shH_{\mathrm{mix}}$, we have  $\nabla_{\xi^{\#}}\delta=0, \forall \xi \in \fot$, in particular $\nabla_{b^{\#}}\delta=0, \forall b \in \fot_{\ZZ}$.
This implies that $f_{b}$ is constant 1 on $\supp \delta$, for any $b\in \fot_{\ZZ}$.
Therefore we conclude that $\delta$ should be supported in the set where $\mu$ takes integral value. That is,
$$\supp \delta \subset \bigcup _{\lambda \in  \fot^{*}_{\ZZ}} \mu^{-1} (\lambda).$$
To prove (2), given any $\lambda \in \fot_{\ZZ}^{*}$ and $\delta \in \shH_{\mathrm{mix}, \lambda}$, we need to show for any $\tau \in \Gamma_{c}(M, L^{-1})$ and $\xi \in \fot$, 
$$
(\xi \cdot \delta)(\tau)= i \langle\lambda, \xi\rangle \delta(\tau),$$ where $\langle, \rangle : \fot^{*} \times \fot \rightarrow \RR$ is the natural pairing. Note that the Liouville volume form $\vol_{M}$ is $T^{n}$-invariant, $\di \xi^{\#}=0$. This implies,
\begin{equation}\label{eq2-7}
(\nabla_{\xi^{\#}}\delta)(\tau) =-\delta(\nabla_{\xi^{\#}} \tau)=0, ~~ \text{and} ~~(\xi \cdot \delta)(\tau)=-\delta(\xi \cdot \tau).
\end{equation}
Recall that $\xi \cdot \tau= \nabla_{\xi^{\#}} \tau + i \mu^{\xi} \tau$.
By equation (\ref{eq2-7}), we have \begin{equation}\label{eq3-1-5}
(\xi \cdot \delta)(\tau_{k})= -\delta(\xi \cdot \tau_{k})= -\delta( \nabla_{\xi^{\#}} \tau_{k} + i\mu^{\xi} \tau_{k})= -\delta( i\mu^{\xi} \tau_{k}).
\end{equation}
Suppose $\tau= \tau_{k} \in \Gamma_{c}(M, L^{-1})$ has weight $k$, i.e. 
 $\xi \cdot \tau_{k}=i\langle k, \xi \rangle \tau_{k}$, by equation (\ref{eq2-7}), one has 
\begin{equation}\label{eq3-1-6}
(\xi \cdot \delta)(\tau_{k})= -\delta(\xi \cdot \tau_{k})= -\delta( i\langle k, \xi \rangle \tau_{k}).
\end{equation}
Combine equations (\ref{eq2-7}) and (\ref{eq3-1-6}), we obtain
\begin{equation}\label{eq3-1-8}
\delta( i(\mu^{\xi}-\langle k, \xi \rangle ) \tau_{k})=0, \forall \xi \in \fot.
\end{equation}

 For any $k \ne \lambda$, there exists $\xi \in \fot$, such that $\langle \lambda, \xi \rangle \ne  \langle k, \xi \rangle$. For such $\xi$, as $\mu^{\xi}|_{M^{\lambda}}= \langle \lambda, \xi \rangle$, $\mu^{\xi}- \langle k, \xi \rangle$ is no-where vanishing on a $T^{n}$-invariant open neighbourhood of $M^{\lambda}$. One has:
\begin{equation}\label{eq3-1-9}
\delta(\tau_{k})=\delta( i(\mu^{\xi} - \langle k, \xi \rangle)  \frac{1}{ i(\mu^{\xi} - \langle k, \xi \rangle)}\tau_{k}).
\end{equation}
Since the moment map is $T^{n}$-invariant, $\frac{1}{ i(\mu^{\xi} - \langle k, \xi \rangle)}\tau_{k}$ still has weight $k$. Hence, by the above discussion (i.e. in equation \ref{eq3-1-8}, replacing $\tau_{k}$ by $\frac{1}{ i(\mu^{\xi} - \langle k, \xi \rangle)}\tau_{k}$), one has:
 \begin{equation} \label{eq3-1-7}
 \delta(\tau_{k})=0.
 \end{equation}

Given the weight decomposition of $\tau= \sum_{k}\tau_{k}$ (i.e. $\xi \cdot \tau_{k}=i \langle k, \xi \rangle \tau_{k}$), by linearity of $\delta$ and equation (\ref{eq3-1-7}),
we obtain:
\begin{align*}(\xi \cdot \delta)(\tau)= i\langle \lambda, \xi \rangle\delta(\tau_{ \langle \lambda, \xi \rangle}) =
i \langle \lambda, \xi \rangle\delta(\tau).
\end{align*}
Therefore we have: $\xi \cdot \delta=i\langle \lambda, \xi \rangle \delta$.
\end{proof}
The next corollary says that any element in $\shH_{\mathrm{mix}, \lambda}$ is locally a delta function along $\mu^{-1}(\lambda)$, and does not involve any derivative of delta functions.
\begin{corollary}\label{co3-3}
For any $\lambda \in \fot^{*}_{\ZZ}$, $\delta \in \shH_{\mathrm{mix}, \lambda}$, and any test section $\tau \in \Gamma_{c}(M, L^{-1})$ satisfying $\tau|_{M^{\lambda}}=0$, we have
\begin{equation}
\delta(\tau)=0.
\end{equation}
\end{corollary}

\begin{proof}
 For any $\tau \in \Gamma_{c}(M, L^{-1})$ satisfying $\tau|_{M^{\lambda}}=0$, let $\tau = \sum_{k} \tau_{k}$ be its weight decomposition, where $\tau_{k}=\int_{e^{it} \in T^{n}}(e^{it}\cdot \tau ) e^{-ikt}dt$.  By Theorem \ref{thm3-0-1}, $\delta$ has weight $\lambda$ with respect to $T^{n}$-action. This implies, for $k \ne \lambda$,
 \begin{align*}
 \delta(\tau_{k})
=0.
\end{align*}
Note that $$\tau_{k}(p)=\int_{e^{it} \in T^{n}}(e^{-it}\cdot \tau )(p) e^{ikt}dt=\int_{e^{it} \in T^{n}} \tau (e^{-it}\cdot p) e^{ikt}dt.$$
For any $p \in M^{\lambda}$ and $t \in \fot$, $e^{it}\cdot p \in M^{\lambda}$ since $M^{\lambda}$ is $T^{n}$-invariant.
  Therefore $\tau_{k}|_{M^{\lambda}}=0$, as $\tau|_{M^{\lambda}}=0$.
In particular, $\tau_{\lambda}|_{M^{\lambda}}=0$. So there exists weight $\lambda$ test section $\tau'$ such that $\tau_{\lambda}=i(\mu^{\xi}- \langle \lambda, \xi \rangle)\tau_{\lambda}'$ for some $\xi \in \fot$.
Hence, for $k=\lambda$, by equation (\ref{eq3-1-8}), we obtain
\begin{equation}
\delta(\tau_{\lambda})=\delta(i(\mu^{\xi}-\langle \lambda, \xi \rangle)\tau'_{\lambda})=0.
\end{equation}
Therefore, by the linearity of $\delta$, we have $\delta(\tau)=0$.
\end{proof}
In order to establish the isomorphism between $H^{0}(M_{\lambda}, L_{\lambda})$ and $ \shH_{\mathrm{mix},\lambda}$, where $(M_{\lambda},L_{\lambda})=(M, L)//_{\lambda}T$ is the symplectic reduction of $M$ at a regular integral level $\lambda$.
We first define the distributional section $\delta^{s}\in \Gamma_{c}(M^{\lambda},(L^{\lambda})^{-1})'$ on $M^{\lambda} \subset M$ associated to $s \in H^{0}(M_{\lambda}, L_{\lambda})$ as follows. 
 \begin{definition}\label{def4-0-1}
For any  $\lambda \in \fot^{*}_{\ZZ, \mathrm{reg}}$ and $s \in H^{0}(M_{\lambda}, L_{\lambda})$, we define the {\em distributional section $\delta^{s} \in \Gamma_{c}(M^{\lambda},(L^{\lambda})^{-1})'$ associated to $s$} as follows: for any $\tau \in \Gamma_{c}(M^{\lambda},(L^{\lambda})^{-1})$,
\begin{equation} \label{eq3-1-13}
\delta^{s} (\tau)= \int_{M^{\lambda}} \left\langle \pi^{*}s,\tau \right\rangle \vol^{\lambda}.
\end{equation}
\end{definition}

In fact, $\delta^{s}=\iota(\pi^{*}s)$, under the embedding
 \begin{tikzcd}
 	\iota: \Gamma(M^{\lambda}, L^{\lambda}) \arrow[hook]{r}{\vol^{\lambda}}& \Gamma_{c}(M^{\lambda}, (L^{\lambda})^{-1})'
 \end{tikzcd} defined by $\sigma \mapsto( \iota \sigma)(\tau) = \int_{M} \langle \sigma, \tau \rangle \vol^{\lambda}
$ with respect to $\vol^{\lambda}$. Note that $\imath(\delta^{s}) \in \Gamma_{c}(M, L^{-1})'$, where $\imath:\Gamma_{c} (M^{\lambda}, (L^{\lambda})^{-1})' \hookrightarrow \Gamma_{c}(M, L^{-1})'$ is a natural inclusion defined by:
\begin{equation}
(\imath(\delta))(\tau)= \delta(\tau|_{M^{\lambda}}),  \forall \delta \in \Gamma_{c} (M^{\lambda}, (L^{\lambda})^{-1})'.
\end{equation}

 In order to show $\imath(\delta^{s}) \in \shH_{\mathrm{mix}, \lambda}$, we first need to extend covariant derivative $\nabla_{\xi}$ on the space of smooth sections to the space of distributional sections of $L^{\lambda}$ with respect to $\vol^{\lambda}$ as before. That is, for any $\sigma \in \Gamma_{c}(M^{\lambda}, (L^{\lambda})^{-1})'$, and $\xi \in \Gamma(M^{\lambda}, TM^{\lambda} \otimes \CC)$, \begin{equation} (\nabla _{\xi} \sigma)(\tau) = \sigma(^{t}\nabla_{\xi} \tau), ~~~~\text{with}~~~~ ^{t}\nabla_{\xi} \tau = - (\di \xi \tau + \nabla_{\xi} \tau),\end{equation}
where $\di^{2} \xi= \frac{\shL_{\xi}\vol^{\lambda}}{\vol^{\lambda}}$.

In particular, $\imath(\delta^{s}) \in \Gamma_{c}(M, L^{-1})'$. 
 we also need to study the relationship between the covariant derivative on the space of distributional sections of $L^{\lambda}$ and the space of distributional sections of $L$.
We expect the following diagram
$$
\begin{tikzcd}
	\Gamma(M, L) \arrow{d} \arrow[hook]{r}{\vol_{M}}& \Gamma_{c}(M, L^{-1})'  \arrow{r}{\nabla_{\xi}}& \Gamma_{c}(M, L^{-1})'\\
	\Gamma(M^{\lambda}, L^{\lambda}) \arrow[hook]{r}{\vol^{\lambda}} &\Gamma_{c}(M^{\lambda}, (L^{\lambda})^{-1})' \arrow[hook]{u}{\imath} \arrow{r}{\nabla_{\xi}}& \Gamma_{c}(M^{\lambda}, (L^{\lambda})^{-1})' \arrow[hook]{u}{\imath},
\end{tikzcd}
$$
commute, for any $\xi \in \Gamma(M, TM \otimes \CC)$ satisfying $\xi|_{M^{\lambda}} \in \Gamma(M^{\lambda}, TM^{\lambda} \otimes \CC)$.
In order to show that the above diagram commute, we use the coisotropic embedding theorem due to Weinstein \cite{We1} and further studied by Guillemin in \cite{Gui1} to (relate $\vol_{M}$ and $\vol^{\lambda}$) show the following lemma. 
\subsubsection{Restriction commutes with taking divergence}
 For any  $\xi \in \Gamma(M, TM \otimes \CC)$, we denote the restriction of the divergence of $\xi$ (with respect to $\vol_{M}$) to $M^{\lambda}$ by $\di_{\xi}^{1}$ and denote the divergence of $\xi|_{M^{\lambda}}$ (with respect to $\vol^{\lambda}$) by $\di_{\xi}^{2}$ i.e. $$\di_{\xi}^{1}=\frac{ d(i_{\xi}  \vol_{M})}{\vol_{M}}|_{M^{\lambda}}, 
~~ \text{and} ~~\di^{2} \xi  = \frac{d(i_{\xi|_{M^{\lambda}}} \vol^{\lambda})}.{\vol^{\lambda}}.$$

\begin{lemma} \label{lem4-0-1}
Under the assumption $(*)$, for any $\xi \in \Gamma(M, \shP_{\mathrm{mix}})$ and $\lambda \in \fot^{*}_{\mathrm{reg}}$, we have $$\di^{1} \xi = \di^{2} \xi,$$
as functions on $M^{\lambda}$.
\end{lemma}
\begin{proof}
Without loss of generality, we assume $\lambda=0$ and $n=1$.
In order to show that $\di^{1} \xi = \di^{2} \xi$,
we shall first relate the volume form $\vol_{M}$ of $M$ and the volume form $\vol_{M^{\lambda}}$ of $M^{\lambda}$.
Taking a principal $T^{1}$-connection $\alpha \in \Omega^{1}(M^{0}, \fot)$ on $M^{0}$, choose a basis $\xi_{1}$ of $\fot$ and denote the corresponding dual basis of $\fot^{*}$ by $\xi_{1}^{*}$ with coordinate function $t$. In terms of $\xi_{1}$, we write
$\alpha= \xi_{1}\otimes \alpha_{1}$, where $\alpha_{1}$ is a scalar valued form. By abuse of notations, we denote $\alpha_{1}$ by $\alpha$.
Consider $M^{0}$ as a submanifold of $M^{0} \times \fot^{*}$ via the embedding $$i_{0}: M^{0} \rightarrow M^{0} \times \fot^{*}, i_{0}(p)=(p,0).$$ 
The two-form $$\tilde{\omega}= \pi^{*} \omega_{0} + d(t \alpha)$$ is symplectic on a neighbourhood $U$ of $ M^{0}$ in $M^{0} \times \fot^{*}$ and satisfies $i_{0}^{*} \tilde{\omega}=\pi^{*}\omega_{0}$.

 Note that $(dt\wedge \alpha)^{2}=0$ and $(i^{*}\omega)^{m}=0$.  
Then we restrict our attention to show $\shL_{\xi} t=0,~~ \forall~ \xi \in \Gamma(M, \shP_{\mathrm{mix}}).$ We extend the $T^{1}$-action on $M^{0}$ to $M^{0} \times \fot^{*}$ in a trivial manner.
Then $\tilde{\omega}$ is $T^{1}$-invariant and the action of $T^{1}$ on $ M^{0} \times \fot^{*}$ is Hamiltonian with moment map
$$\mu_{0}:  M^{0} \times \fot^{*} \rightarrow \fot^{*}, (p, t) \mapsto t.$$ 
According to Theorem 2.2 of \cite{Gui1}, in a neighborhood $U$ of $ M^{0}$, the Hamiltonian $T^{1}$-spaces $(M, \omega)$ and $ (M^{0} \times \fot^{*}, \tilde{\omega})$ are isomorphic (see Appendix).
This gives $\shL_{\zeta} t=0$, for any $\zeta \in \Gamma(U,\shD_{\CC})$ and 
 $$\vol_{M} =\frac{1}{m!} (i^{*}\omega + td\alpha)^{m-1} \wedge \alpha \wedge dt,$$
 in a neighbourhood $U$ of $M^{0}$. As $\shP_{\mathrm{mix}} \subset \shD_{\CC}$, it is obvious that $\Gamma(U, \shP_{\mathrm{mix}}) \subset \Gamma(U,\shD_{\CC})$. This gives us that, for any $ \xi \in \Gamma(U, \shP_{\mathrm{mix}})$,
 \begin{equation}\shL_{\xi} t=0. \end{equation} 
It follows that:
$$
\shL_{\xi} (i^{*}\omega + td\alpha)^{m-1}=(m-1)(\shL_{\xi} (i^{*}\omega) + t\shL_{\xi} d\alpha) \wedge(i^{*}\omega + td\alpha)^{m-2},~ \forall \xi \in \Gamma(U, \shP_{\mathrm{mix}}).
$$

Therefore, we obtain:
\begin{align*}
&\frac{1}{\vol_{M}} d(i_{\xi} \vol_{M})=\frac{1}{\vol_{M}} \shL_{\xi} \vol_{M}=\frac{1}{\vol_{M}}\frac{1}{m!} \shL_{\xi}( (i^{*}\omega + td\alpha)^{m-1}\wedge \alpha \wedge dt) \\
&=\frac{1}{\vol_{M}}\frac{1}{m!} (\shL_{\xi} (i^{*}\omega + td\alpha)^{m-1}\wedge \alpha \wedge dt +(i^{*}\omega + td\alpha)^{m-1} \wedge \shL_{\xi}\alpha \wedge dt) \\
&=\frac{(m-1)(\shL_{\xi} i^{*}\omega + t\shL_{\xi} d\alpha) \wedge(i^{*}\omega + td\alpha)^{m-2}\wedge \alpha  +(i^{*}\omega + td\alpha)^{m-1} \wedge \shL_{\xi}\alpha )
}{ (i^{*}\omega + td\alpha)^{m-1} \wedge \alpha }.
\end{align*}
Recall that $\vol^{0}= \frac{1}{(m-1)!} (i^{*}\omega )^{m-1} \wedge \alpha$. By abuse of notation, $i_{\xi} \vol^{0}$ means $i_{\xi|_{M^{0}}} \vol^{0}$. Then by a straight computation,
\begin{align*}
\di^{2} \xi &= \frac{1}{\vol^{0}} d(i_{\xi} \vol^{0})=\frac{1}{\vol^{0}} \shL_{\xi} \vol^{0}\\
&=\frac{1}{\vol^{0}}\frac{1}{(m-1)!} \shL_{\xi}( (i^{*}\omega )^{m-1}\wedge \alpha ) \\
&=\frac{(m-1)(\shL_{\xi} i^{*}\omega ) \wedge(i^{*}\omega )^{m-2}\wedge \alpha  +(i^{*}\omega )^{m-1} \wedge \shL_{\xi}\alpha )}{(i^{*}\omega )^{m-1} \wedge \alpha}. 
\end{align*}
Therefore, for $\xi \in \Gamma(M, \shP_{\mathrm{mix}})$, we have:
$$\di^{1} \xi =\left(\frac{1}{\vol_{M}} d\left(i_{\xi} \vol_{M}\right)\right)|_{M^{0}}=\left(\frac{1}{\vol_{M}} d\left(i_{\xi} \vol_{M}\right)\right)|_{t=0}=\di^{2} \xi.$$
\end{proof}

Then we obtain the following theorem:
\begin{theorem} \label{thm3-5}For any $\lambda \in \fot^{*}_{\ZZ, \mathrm{reg}} $, $\delta \in \Gamma_{c}(M^{\lambda}, (L^{\lambda})^{-1})' $ and $\xi \in \Gamma(M, TM \otimes \CC)$ satisfying $\xi|_{M^{\lambda}} \in \Gamma(M^{\lambda}, TM^{\lambda} \otimes \CC)$, we have
\begin{equation}\label{eq2-1}
\nabla_{\xi}(\imath(\delta))=\imath(\nabla_{\xi}\delta),
\end{equation}
where $\imath:\Gamma_{c} (M^{\lambda}, (L^{\lambda})^{-1})' \hookrightarrow \Gamma_{c}(M, L^{-1})' $ is the natural inclusion.
\end{theorem}

 \begin{proof}
 For any test section $\tau \in \Gamma_{c}(M,L^{-1})$, according to equation (\ref{eqcon}), one has
 \begin{equation}\label{eq2-0}
(\nabla_{\xi}(\imath(\delta)))(\tau)= \imath(\delta)( ^{t}\nabla_{\xi} \tau )=\delta( i^{*} (^{t}\nabla_{\xi} \tau) ),\end{equation}
 and 
 \begin{equation}\label{eq2-0-0}
(\imath(\nabla_{\xi}\delta))(\tau)=(\nabla_{\xi}\delta)(i^{*}\tau)=\delta(^{t}\nabla_{\xi} (i^{*} \tau)),\end{equation}
where $i: M^{\lambda} \hookrightarrow M$ is the inclusion. 
 
   To show equation (\ref{eq2-1}), by (\ref{eq2-0}) and (\ref{eq2-0-0}), it is enough to prove that
\begin{equation}\label{eq2-2}
i^{*} (^{t}\nabla_{\xi} \tau) = ^{t}\nabla_{\xi} (i^{*} \tau).
\end{equation}

According to equation (\ref{eqcon}), we have:
\begin{equation}\label{eq2-3}
i^{*}(^{t}{\nabla}_{\xi} \tau)= -i^{*} \left(\di \xi \tau+ \nabla_{\xi}\tau\right),
\end{equation}
where $\di \xi =\frac{ i_{\xi}  \vol_{M}}{\vol M}$.
Similarly, applying the equation (\ref{eqsec}) to $L|_{M^{\lambda}}$, we have:
\begin{equation}\label{eq2-4}
^{t}\nabla_{\xi} (i^{*} \tau)= -((\di^{2} \xi )(i^{*}\tau) + \nabla_{\xi}(i^{*}\tau))
\end{equation}
where $\di^{2} \xi  = \frac{i_{\xi} \vol^{\lambda}}{\vol^{\lambda}}$, $i^{*}\tau=\tau|_{M^{\lambda}}=\tau$ by abuse of notation. 
Denote $i^{*}(\di \xi)$ by $\di^{1} \xi$ i.e. $\di^{1}\xi =\frac{ i_{\xi}  \vol_{M}}{\vol_{ M}}|_{M^{\lambda}}$.   
As $i^{*}(\nabla_{\xi}\tau) =\nabla_{\xi}(i^{*}\tau)$ by abuse of notation $\xi|_{M^{\lambda}}=\xi$, we have
\begin{equation}\label{eq2-5}  -i^{*} \left(\di \xi \tau+ \nabla_{\xi}\tau\right)=-(\di^{1}\xi(i^{*}\tau) +\nabla_{\xi}(i^{*}\tau)).
\end{equation}
 By Lemma \ref{lem4-0-1}, 
 \begin{equation}\label{eq2-6}
\di^{1}\xi=\di^{2}\xi. \end{equation}

  Combining equations (\ref{eq2-3}), (\ref{eq2-4}), (\ref{eq2-5}) with (\ref{eq2-6}), one has
$$\delta( i^{*} (^{t}\nabla_{\xi} \tau) )= \delta(^{t}\nabla_{\xi} (i^{*} \tau)).$$  
Therefore we have: $\nabla_{\xi}(\imath(\delta))=\imath(\nabla_{\xi}\delta)$.
 \end{proof}

\begin{proposition}\label{pro4-0-4}
For any regular $\lambda \in \fot^{*}_{\ZZ, \mathrm{reg}} $ and $s \in H^{0}(M_{\lambda}, L_{\lambda})$, we have: 
$$\imath(\delta^{s}) \in \shH_{\mathrm{mix},\lambda},$$
where $\imath:\Gamma_{c} (M^{\lambda}, (L^{\lambda})^{-1})' \hookrightarrow \Gamma_{c}(M, L^{-1})' $ is the natural inclusion.
\end{proposition}

\begin{proof}

 By the definition of $\delta^{s}$, we have $\imath(\delta^{s}) \in \Gamma_{c}(M,L^{-1})'$ and $\supp \imath(\delta^{s}) \subset \mu^{-1}(\lambda)$. It remains to show that, for any $\xi \in \Gamma(M, \shP_{\mathrm{mix}})$,
 \begin{equation} \label{eq3-11}\nabla_{\xi} (\imath(\delta^{s}))= 0.\end{equation}
Note that for any $\xi \in (M, \shP_{\mathrm{mix}})$, $\xi|_{M^{\lambda}} \in \Gamma(M^{\lambda}, TM^{\lambda} \otimes \CC)$. 
To check equation (\ref{eq3-11}), by Theorem \ref{thm3-5}, it is equivalent to prove, for any $\xi \in \Gamma(M^{\lambda}, \shP_{\mathrm{mix}})$
\begin{equation}
\nabla_{\xi} \delta^{s}=0.
\end{equation}

  Take any test section $\tau \in \Gamma_{c}(M^{\lambda},(L^{\lambda})^{-1})$, according to equation (\ref{eqcon}), we have:
\begin{equation}\label{eq3-12}
(\nabla_{\xi} \delta^{s})(\tau) = \delta^{s}\left(^{t}{\nabla}_{\xi} \tau\right)=- \delta^{s}\left((\di^{2} \xi) \tau+ \nabla_{\xi}\tau\right),
\end{equation}
where $\di^{2} \xi  = \frac{i_{\xi} \vol^{\lambda}}{\vol^{\lambda}}$.
 By definition of $\delta^{s}$, it can be seen that:
\begin{equation}\label{eq3-14}
-\delta^{s}\left((\di^{2} \xi ) \tau + \nabla_{\xi}\tau\right)
= -\int_{M^{\lambda}} \left\langle\pi^{*}s, (\di^{2} \xi) \tau + \nabla_{\xi}\tau\right\rangle \vol^{\lambda}.
\end{equation}
Similarly, applying the equation (\ref{eqsec}) to $L|_{M^{\lambda}}$, we have:
\begin{equation}\label{eq3-15}
 \int _{M^{\lambda}}\left\langle \nabla_{\xi}(\pi^{*}s), \tau \right\rangle \vol^{\lambda}=-\int_{M^{\lambda}}\left\langle\pi^{*}s, (\di^{2} \xi )\tau + \nabla_{\xi}\tau \right\rangle \vol^{\lambda}.\end{equation}

  Combining equations (\ref{eq3-12}), (\ref{eq3-14}), with (\ref{eq3-15}), we have
 \begin{equation}\label{eq3-8}
(\nabla_{\xi} \delta^{s})(\tau)=\int _{M^{\lambda}}\left\langle \nabla_{\xi}(\pi^{*}s), \tau \right\rangle \vol^{\lambda}.
\end{equation}
  
Since $s \in H^{0}(M_{\lambda}, L_{\lambda})$ is a holomorphic section, we have $\nabla s \in \Gamma(M_{\lambda}, T^{*}M_{\lambda}^{1,0}\otimes L_{\lambda})$. 
For any $\xi \in \Gamma(M, \shP_{\mathrm{mix}})$ and $q \in M^{\lambda}$, as $(\shP_{\mathrm{mix}})_{q} \subset (\shD_{\CC})_{q}= T_{q}M^{\lambda} \otimes \CC$, we have $\pi_{*}(\xi_{q}) \in T_{\pi(q)}M_{\lambda}^{0,1}$. 
 This implies $\nabla_{\xi} (\pi^{*}s) =0$ on $M^{\lambda}$, for any $\xi \in \Gamma(M, \shP_{\mathrm{mix}})$. Then, for all $\tau \in \Gamma_{c}(M,L^{-1})$, by equation (\ref{eq3-8}),
 \begin{align*}
(\nabla_{\xi} \delta^{s})(\tau) =0.
\end{align*} 
Therefore we have: $\imath(\delta^{s} )\in \shH_{\mathrm{mix},\lambda}.$
\end{proof}

\subsection{$\lambda$-weight quantum subspace $\shH_{\mathrm{mix}, \lambda}$ }
In this subsection, we are going to show that (see Theorem \ref{thm4-0-11}) for any regular $\lambda \in \fot^{*}_{\ZZ, \mathrm{reg}}$, $$ \kappa: H^{0}(M_{\lambda}, L_{\lambda}) \rightarrow \shH_{\mathrm{mix},\lambda}$$ given by
$s \mapsto \kappa(s)=\imath(\delta^{s})$ is an isomorphism.

Firstly, we show that $T^{n}$-invariant distributional sections of $L^{\lambda}$ can be descended to distributional sections of $L_{\lambda}$. That is, for any $\delta \in \Gamma_{c} (M^{\lambda}, (L^{\lambda})^{-1})'$ satisfying $\nabla_{\xi^{\#}} \delta=0$, there exists a distributional section $\eta \in \Gamma_{c}(M_{\lambda}, L_{\lambda}^{-1})'$ such that $\delta= \pi^{*} \eta$ (Lemma \ref{lem4-0-6}). Secondly, we show that if $\nabla_{\zeta}( \pi^{*} \eta)=0$ for all $\zeta \in \Gamma(M^{\lambda}, \shP_{\mathrm{mix}})$, then $\eta$ is $\bar{\partial}$-closed (Theorem \ref{thm4-0-8}). Finally, we show that $H^{0}(M_{\lambda}, L_{\lambda}) \cong \shH_{\mathrm{mix},\lambda}$ (Theorem \ref{thm4-0-11}).

\subsubsection{Descending distributional sections from $M^{\lambda}$ to $M_{\lambda}$} 
 For any $\lambda \in \fot^{*}_{\mathrm{reg}}$, let $\pi:M^{\lambda} \rightarrow M_{\lambda}$ be the principal $T^{n}$-bundle. Recall that $(L^{\lambda}, \nabla)$ can be descended to $M_{\lambda}$ which we denote as $(L_{\lambda}, \nabla)$. 
 According to Remark \ref{rk4-1}, we have $\pi_{*}: \Gamma(M^{\lambda}, (L^{\lambda})^{-1}) \rightarrow \Gamma(M_{\lambda},L^{-1}_{\lambda})$ and dually we have $\pi^{*}: \Gamma_{c}(M_{\lambda},L^{-1}_{\lambda})' \rightarrow \Gamma_{c}(M_{\lambda},(L^{\lambda})^{-1})'$.
 
In fact, our above claim $\delta= \pi^{*}\eta$ holds true for any $T^{n}$-principal bundle $P \rightarrow B$. 
Let $\pi: P \rightarrow B$ be a principal $T^{n}$-bundle with a fiberwise $T^{n}$-invariant volume form $d\theta$ such that $\int_{P}d\theta=1 \in C^{\infty}(B)$. 
Let $(E, \nabla)$ be a line bundle over $B$. We can push-forward sections of $\pi^{*}E$ to sections of $E$ with respect to $d \theta$. Furthermore we have:
\begin{lemma}\label{lem4-0-6}
 Taking $\delta \in \Gamma_{c}(P, (\pi^{*}E)^{-1})'$,
if $\nabla_{\xi^{\#}} \delta=0$ for any $\xi \in \fot$, then there exists a distributional section $\eta \in \Gamma_{c}(B,E^{-1})'$ such that 
$$\delta= \pi^{*}\eta.$$ 
\end{lemma}

\begin{proof} 
By partition of unity, it is enough to show that on any open subset $U$ of $B$, for $\delta \in \Gamma_{c}(\pi^{-1}(U), (\pi^{*}E)^{-1})'$,
if $\nabla_{\xi^{\#}} \delta=0$ for any $\xi \in \fot$, there exists a distributional section $\eta \in \Gamma_{c}(U,E^{-1})'$ such that $\delta= \pi^{*}\eta$. That is, for any $\tau \in \Gamma_{c}(\pi^{-1}(U), (\pi^{*}E)^{-1})$, 
$$ \delta(\tau) =\eta( \pi_{*}\tau).$$

Fixing a local frame $\sigma_{0} \in \Gamma(U,E)$ of $E$ on an open subset $U \subset B$, let $\sigma:=\pi^{*}\sigma_{0}$ and $\sigma^{-1}$ be the corresponding local frames of $\pi^{*}E$ and $(\pi^{*}E)^{-1}$ respectively on $\pi^{-1}(U)$. 
With respect to local frames $\sigma$ and $\sigma^{-1}$, the distributional section $\delta \in \Gamma_{c}(\pi^{-1}(U), (\pi^{*}E)^{-1})'$ corresponds to the distributional function
$f_{\delta} \in \Gamma_{c}(\pi^{-1}(U),\CC)'$, where $f_{\delta}$ is determined by: 
\begin{equation} \label{eq3-3-3} f_{\delta}(g_{\tau})=\delta(g_{\tau} \sigma^{-1}),
\end{equation}
 for any text function $g_{\tau} \in \Gamma_{c}(\pi^{-1}(U),\CC)$.
We restrict our attention to show that $ \nabla_{\xi^{\#}}\delta =0 $ if and only if $ \xi^{\#}f_{\delta}=0$. Applying the equation (\ref{eqcon}) to line bundle $\pi^{*}L$ and trivial bundle over $\pi^{-1}(U)$ respectively, it can be seen that:
\begin{equation}  \label{eq3-3-4}\left(\nabla_{\xi^{\#}}\delta \right)\left( \tau\right)
 =-\delta \left((\di \xi^{\#}) \tau + \nabla_{\xi^{\#}} \tau\right),
 \end{equation} 
 and
\begin{equation}\label{eq3-3-7}
( \xi^{\#} f_{\delta})\left( g_{\tau}\right)= f_{\delta} \left(-\left(\di \xi^{\#} g_{\tau}+\xi^{\#}g_{\tau}\right)\right).
\end{equation}
 Since $\nabla_{\xi^{\#}}\sigma=\nabla_{\xi^{\#}}(\pi^{*}\sigma_{0}) =0$, one has $\nabla_{\xi^{\#}}\sigma^{-1}=0$ and 
 \begin{equation}  \label{eq3-3-5}\nabla_{\xi^{\#}} \tau=\nabla_{\xi^{\#}}\left(g_{\tau}\sigma^{-1}\right)= \left(\xi^{\#}g_{\tau}\right)\sigma^{-1}.
 \end{equation}
Combining equations ( \ref{eq3-3-3}), ( \ref{eq3-3-4}), (\ref{eq3-3-7}), with (\ref{eq3-3-5}), we obtain that
\begin{equation}\label{eq3-3-8} \left(\nabla_{\xi^{\#}}\delta \right)\left( \tau\right)= (\xi^{\#} f_{\delta})\left( g_{\tau}\right),\end{equation}
 for any $\tau \in \Gamma_{c}\left(\pi^{-1}(U), (\pi^{*}E)^{-1}\right)$.
It turns out that $\nabla_{\xi^{\#}} \delta=0$ iff $\xi^{\#}f_{\delta}=0$ for any $\xi \in \fot$. Then by Lemma \ref{lem4-0-5}, there exists a distributional function $f_{\eta} \in \Gamma_{c}(U,\CC)'$ such that $f_{\delta}=\pi^{*}(f_{\eta})$. 
Define $\eta \in  \Gamma_{c}(U, (\pi^{*}E)^{-1})'$ to be distributional section associated to $f_{\eta}$ with respect to the nowhere vanishing section $\sigma^{-1}_{0}$, that is $\eta(h_{\tau} \sigma_{0}^{-1})=f_{\eta}(h_{\tau})$. For any test section $\tau \in \Gamma_{c}(\pi^{-1}(U), \pi^{*}E)$, it can be check that: $$\delta(\tau)=
(\pi^{*}\eta)\tau.$$
Therefore we have $\delta =\pi^{*} \eta$. 
\end{proof}

\begin{lemma} \label{lem4-0-5}
Let $\pi: P \rightarrow B$ be the principal $T^{n}$-bundle and let $U$ be any open subset of $B$. Let $\delta \in \Gamma_{c}(\pi^{-1}(U),\CC)'$ be a distributional function. If $\xi^{\#}\delta =0$ for any $\xi \in \fot$, there exists a distributional function $\eta \in  \Gamma_{c}(U,\CC)'$, such that $\delta=\pi^{*}\eta$. Namely,
$$\delta( g)= \eta( \pi_{*}g),  \forall ~g \in  \Gamma_{c}(\pi^{-1}(U),\CC).$$
\end{lemma}

\begin{proof}
For any $\delta \in \Gamma_{c}(\pi^{-1}(U),\CC)'$,
there exist $\delta_{\epsilon} \in \Gamma(\pi^{-1}(U),\CC)$ (see \cite{ HJ, Ru}) such that $\lim_{\epsilon \rightarrow 0} \delta_{\epsilon} = \delta$ and 
\begin{equation}(\xi^{\#}\delta_{\epsilon})( g ) = (\xi^{\#}\delta)( g_{\epsilon}),\end{equation}
 for any $g \in  \Gamma_{c}(\pi^{-1}(U),\CC)$.
As $\xi^{\#}\delta=0$, we obtain $\xi^{\#}\delta_{\epsilon}=0$.
Since $\delta_{\epsilon}$ is smooth, there exists a smooth function $\eta_{\epsilon} \in \Gamma(U,\CC)$, such that $\delta_{\epsilon}= \pi^{*}\eta_{\epsilon} \in \Gamma(\pi^{-1}(U),\CC)$.
It can be check that  
$$
\lim_{\epsilon \rightarrow 0} \eta_{\epsilon}(h)=  \lim_{\epsilon \rightarrow 0}\delta_{\epsilon}( \pi^{*}h),
$$
for any $h \in  \Gamma_{c}(U,\CC)$. Hence we have $\lim_{\epsilon \rightarrow 0}\eta_{\epsilon}$ exists and denoted by $\eta$. It follows $$\delta=\pi^{*}\eta.$$
\end{proof}

\subsubsection{Pulling back commutes with taking divergence}
 Fix $\lambda \in \fot^{*}_{\ZZ, \mathrm{reg}}$, let $\alpha \in \Omega^{1}(M^{\lambda},\fot)$ be a connection on the principal $T^{n}$-bundle $ \pi: M^{\lambda} \rightarrow M_{\lambda}$. For any $\zeta \in \Gamma(M_{\lambda}, TM_{\lambda})$, the horizontal lifting of $\zeta$ with respect to $\alpha$ is denoted by $\tilde{\zeta}$. Denote the divergence of $\zeta$ on $M_{\lambda}$ with respect to $\vol_{\lambda}$ by $\di \zeta$ (i.e. $\di \zeta = \frac{\shL_{\zeta} \vol_{\lambda}}{\vol_{\lambda}}$) and denote the divergence of $\tilde{\zeta}$ on $M^{\lambda}$ with respect to $\vol^{\lambda}$ by $\di \tilde{\zeta}$ (i.e. $\di \tilde{\zeta} = \frac{\shL_{\tilde{\zeta}} \vol^{\lambda}}{\vol^{\lambda}}$).

\begin{lemma}\label{lem4-0-7} Let $\di \zeta$ and $\di \tilde{\zeta}$ be defined as above.   Then we have
$$\pi^{*} (\di{\zeta})= \di{\tilde{\zeta}},$$ 
as smooth functions on $M^{\lambda}$.
\end{lemma}

\begin{proof}
 As $T^{n}$ is abelian, the horizontal lifting $\tilde{\zeta}$ of $\zeta$ with respect to the connection one form $\alpha$ is $T^{n}$-invariant. That is
\begin{equation} \label{eq3-0-0}\shL_{\xi^{\#}} \tilde{\zeta} =0, \end{equation}
 for all $\xi \in \fot$, where $\xi^{\#}$ is the fundamental vector field associate to $\xi$.
According to the property of principal $T^{n}$-connection and equation (\ref{eq3-0-1}), we have
\begin{equation} \label{eq3-0-1}(\shL_{\tilde{\zeta}} \alpha)(\xi^{\#})= \shL_{\tilde{\zeta}} (\alpha(\xi^{\#}))- \alpha(\shL_{\tilde{\zeta}} \xi^{\#})=0.\end{equation}
Recall that $\vol^{\lambda}=\pi^{*}\vol_{\lambda} \wedge \alpha^{n}$. By equation $(\ref{eq3-0-1})$, one has 
\begin{equation} \label{eq3-0-2}\shL_{\tilde{\zeta}}\vol^{\lambda}=(\shL_{\tilde{\zeta}}(\pi^{*}\vol_{\lambda}))\wedge \alpha^{n} .\end{equation}
On the other hand, by Cartan formula and $\vol_{\lambda}$ being the volume form on $B$, we have: 
\begin{equation} \label{eq3-0-3}
\shL_{\tilde{\zeta}}(\pi^{*}\vol_{\lambda})= d(i_{\tilde{\zeta}} (\pi^{*}\vol_{\lambda}))= \pi^{*}(\shL_{\zeta}\vol_{\lambda}),
\end{equation}
Recall that
\begin{equation}\label{eq3-0-4}
\shL_{\zeta}\vol_{\lambda}=(\di{\zeta}) \vol_{\lambda},~~ \shL_{\tilde{\zeta}} \vol^{\lambda}= (\di{\tilde{\zeta}}) \vol^{\lambda}.\end{equation}
Combining equation (\ref{eq3-0-2}), (\ref{eq3-0-3}), with (\ref{eq3-0-4}), one has 
\begin{align*} (\di{\tilde{\zeta}}) \vol^{\lambda}&=\shL_{\tilde{\zeta}} \vol^{\lambda}= \pi^{*}(\shL_{\zeta}\vol_{\lambda}) \wedge \alpha^{n} \\
&= \pi^{*}(\di{\zeta}) \pi^{*}\vol_{\lambda} \wedge \alpha^{n}\\
&= \pi^{*}(\di{\zeta}) \vol^{\lambda}.
\end{align*}
Therefore we obtain: $\pi^{*} (\di{\zeta})= \di{\tilde{\zeta}}.$
\end{proof}

\begin{theorem}\label{thm4-0-8}
For any $\lambda \in \fot^{*}_{\ZZ, \mathrm{reg}}$ and distributional function $\eta \in \Gamma_{c}(M_{\lambda}, \CC)' $, if $\nabla_{\xi}(\pi^{*}\eta)=0$, for any $\xi \in \Gamma(M^{\lambda}, \shP_{\mathrm{mix}})$, then we have $\nabla_{\zeta}\eta=0$, for all $\zeta \in \Gamma(M_{\lambda},TM_{\lambda}^{0,1})$.
\end{theorem}

\begin{proof} To prove this statement, fixing the connection one form $\alpha \in \Omega^{1}(M^{\lambda},\fot)$ on principal $T^{n}$-bundle $\pi: M^{\lambda} \rightarrow M_{\lambda}$,
we denote the horizontal lifting of $\zeta$ with respect to the connection $\alpha$ by $\tilde{\zeta}$, for any $\zeta \in \Gamma(M_{\lambda},TM_{\lambda}^{0,1})$. In order to show
$\nabla_{\zeta}\eta= 0$, it is enough to show
, for any test function $ \phi \in D_{c}(M_{\lambda})$, 
$$ \left(\nabla_{\zeta}\eta\right)\left(\phi\right)= \left(\nabla_{\tilde{\zeta}} \left(\pi^{*}\eta\right)\right)\left(\pi^{*}\phi\right).
$$ 
Let $\vol_{\lambda}$ and $\vol^{\lambda}$ be volume forms of $M_{\lambda}$ and $M^{\lambda}$ respectively as defined before. In particular, $\vol^{\lambda}=\pi^{*}\vol_{\lambda} \wedge \alpha^{n}$ with respect to the principal $T^{n}$-connection $\alpha \in \Omega^{1}(M^{\lambda},\fot)$.
Applying equation (\ref{eqcon}) to the trivial bundle of $M_{\lambda}$ and $M^{\lambda}$ respectively, we obtain: 
\begin{equation}\label{eq3-1-0}
\left(\nabla_{\zeta}\eta\right)\left(\phi\right)= 
\eta\left(-\left(\di\zeta
\right)\phi - \nabla_{\zeta}\phi \right),\end{equation}
and 
\begin{equation}\label{eq3-1-1}
 \left(\nabla_{\tilde{\zeta}} \left(\pi^{*}\eta\right)\right)\left(\pi^{*}\phi\right)=
 \left(\pi^{*}\eta\right) \left(-\left( \di\tilde{\zeta}\right)\pi^{*}\phi - \nabla_{\tilde{\zeta}}\left(\pi^{*}\phi\right)\right),
\end{equation}
where $\di \zeta$ ($\di \tilde{\zeta}$ resp.) is the divergence of $\zeta$ ($\tilde{\zeta}$ resp.) with respect to $\vol_{\lambda}$ ($\vol^{\lambda}$ resp.). 
According to the Remark \ref{rk4-1}, we have:
 \begin{equation} \label{eq3-1-2}\left(\pi^{*}\eta\right) \left(\pi^{*}\left(-\left(\di\zeta
\right)\phi - \nabla_{\zeta}\phi\right)\right)
=\eta\left(-\left(\di\zeta
\right)\phi - \nabla_{\zeta}\phi\right).
\end{equation}
By Lemma \ref{lem4-0-7}, 
\begin{equation}\label{eq3-1-3}
\pi^{*} (\di{\zeta})= \di{\tilde{\zeta}}.\end{equation}
Note that $\pi^{*}(\zeta(\phi))= \pi^{*} \zeta(\pi^{*}\phi)$. By equation (\ref{eq3-1-3}), one has
\begin{equation}\pi^{*}\left(-\left(\di\zeta
\right)\phi - \nabla_{\zeta}\phi\right)=-\left( \di{\tilde{\zeta}}\right)\pi^{*}\phi -\nabla_{\tilde{\zeta}}\left(\pi^{*}\phi\right).\end{equation}
Furthermore:
\begin{equation}\label{eq3-1-4}
 \left(\pi^{*}\eta\right)\left(\pi^{*}\left(-\left(\di\zeta
\right)\phi - \nabla_{\zeta}\phi\right)\right)=\left(\pi^{*}\eta\right) \left(-\left( \di{\tilde{\zeta}}\right)\pi^{*}\phi -\nabla_{\tilde{\zeta}}\left(\pi^{*}\phi\right)\right).\end{equation}
Combining (\ref{eq3-1-0}), (\ref{eq3-1-1}), (\ref{eq3-1-2}), with (\ref{eq3-1-4}), we are able to conclude:
\begin{equation}\label{eq3-1-5}  \left(\nabla_{\zeta}\eta\right)\left(\phi\right)= \left(\nabla_{\tilde{\zeta}} \left(\pi^{*}\eta\right)\right)\left(\pi^{*}\phi\right).
\end{equation}
Then we restrict our attention to show $\tilde{\zeta} \in \Gamma(M^{\lambda}, \shP_{\mathrm{mix}})$. 
As $T^{n}$ acts freely on $M^{\lambda}$, $M^{\lambda} \times \fot \cong \shI_{\RR}|_{M^{\lambda}}$. 
Note that $\pi_{*}(\tilde{\zeta})=\zeta \in \Gamma(M_{\lambda}, TM_{\lambda}^{0,1})$ and $\alpha(\tilde{\zeta})=0$.
Since $ T_{p}M^{\lambda}  \otimes \CC \subset (\shD_{\CC})_{p}$ and $(\shP_{\mathrm{mix}})_{p} = (\shD_{\CC} \cap TM^{0,1})_{p} \oplus (\shI_{\CC})_{p}$, for any $p\in M^{0}$, we have $\tilde{\zeta}\in  \Gamma(M^{\lambda}, \shP_{\mathrm{mix}})$. According to what we assume, we have $\nabla_{\tilde{\zeta}} \left(\pi^{*}\eta\right)=0$. Therefore, by equation (\ref{eq3-1-5}), we have $$\left(\nabla_{\zeta}\eta\right)\left(\phi\right)=\left(\nabla_{\tilde{\zeta}} \left(\pi^{*}\eta\right)\right)\left(\pi^{*}\phi\right)=0, ~~~\forall \phi \in D_{c}(M_{\lambda}), \zeta \in \Gamma(M_{\lambda}, TM_{\lambda}^{0,1}) .$$
\end{proof}

\subsubsection{Building the isomorphism $H^{0}(M_{\lambda}, L_{\lambda}) \cong \shH_{\mathrm{mix},\lambda}$}
Recall given any $s \in H^{0}(M_{\lambda},L_{\lambda})$, by Proposition \ref{pro4-0-4}, the associated distributional section $\imath(\delta^{s})$ belongs to $\shH_{mix,\lambda}$. Therefore we can define a homomorphism $$ \kappa: H^{0}(M_{\lambda}, L_{\lambda}) \rightarrow \shH_{\mathrm{mix},\lambda}$$ given by
$s \mapsto \kappa(s)=\imath(\delta^{s})$, where $\imath:\Gamma_{c} (M^{\lambda}, (L^{\lambda})^{-1})' \hookrightarrow \Gamma_{c}(M, L^{-1})' $ is the natural inclusion. It can be checked that $\kappa$ is injective.

\begin{theorem}\label{thm4-0-11}
 For any $\lambda \in \fot^{*}_{\ZZ, \mathrm{reg}}$, $ \kappa: H^{0}(M_{\lambda}, L_{\lambda}) \rightarrow \shH_{\mathrm{mix},\lambda}$ is an isomorphism.
\end{theorem}

\begin{proof} Given any $\tilde{\delta} \in \shH_{\mathrm{mix}, \lambda}$, we need to construct $s \in H^{0}(M_{\lambda}, L_{\lambda})$ such that $\tilde{\delta}=\kappa(s)$. Firstly we show that, there exists $\delta \in \Gamma_{c}(M^{\lambda}, (L^{\lambda})^{-1})'$ such that $\tilde{\delta}= \imath(\delta)$ as follows: we define the distributional section $\delta \in \Gamma_{c}(M^{\lambda}, (L^{\lambda})^{-1})'$ by: 
$$\delta(\tau) = \tilde{\delta}(\tilde{\tau}),$$ for any $\tau \in \Gamma_{c}(M^{\lambda}, (L^{\lambda})^{-1})$, where $\tilde{\tau} \in \Gamma_{c}(M, L^{-1})$ is any test section satisfying $\tilde{\tau}|_{M^{\lambda}}=\tau$
By Corollary \ref{co3-3}, $\delta$ is well defined. 
 Moreover, one has 
 \begin{equation}
 \tilde{\delta}= \imath(\delta).
 \end{equation}
 That is, for any test section $\tilde{\tau}' \in \Gamma_{c}(M, L^{-1})$,
$(\imath(\delta))(\tilde{\tau}')= \delta(\tilde{\tau}'|_{M^{\lambda}}) = \tilde{\delta}(\tilde{\tau}').$
Secondly we show that there exists $\eta \in \Gamma_{c}(M_{\lambda}, L_{\lambda}^{-1})'$ such that $\delta= \pi^{*} \eta$, where $\pi: M^{\lambda} \rightarrow M_{\lambda}$ is the projection.
 For any $\tilde{\delta} \in \shH_{\mathrm{mix},\lambda}$, since $\xi^{\#} \in \Gamma(M, \shP_{\mathrm{mix}})$, we have 
 $\nabla_{\xi^{\#}}\tilde{\delta}=0$, for any $\xi \in \fot$.  By Theorem \ref{thm3-5}, one has
 \begin{equation}
0= \nabla_{\xi^{\#}}\tilde{\delta}=\nabla_{\xi^{\#}}(\imath(\delta)) = \imath( \nabla_{\xi^{\#}}\delta), ~~\forall \xi \in \fot.
 \end{equation}
 By the injectivity of $\imath$, we obtain, for any $\xi \in \fot$,
 \begin{equation}
 \nabla_{\xi^{\#}}\delta=0.
 \end{equation}
 According to Lemma \ref{lem4-0-6}, there exists a distributional section $\eta \in \Gamma_{c}(M_{\lambda},L_{\lambda}^{-1})'$, such that 
 \begin{equation}\label{eq3-48}
 \delta=\pi^{*} \eta.\end{equation}
Next we show that there exists a holomorphic section $s \in H^{0}(M_{\lambda},L_{\lambda})$ such that 
 $\eta= \iota(s)$ under the inclusion map
 $\iota: \Gamma(M_{\lambda}, L_{\lambda}) \rightarrow \Gamma_{c}(M_{\lambda},L_{\lambda}^{-1})'$ with respect to $\vol_{\lambda}$.
 By the definition of $\shP_{\mathrm{mix}}$, for any $\xi \in \Gamma(M, \shP_{\mathrm{mix}})$, we have $\xi|_{M^{\lambda}} \in \Gamma(M^{\lambda}, \shP_{\mathrm{mix}}) \subset  \Gamma(M^{\lambda},TM^{\lambda}\otimes \CC)$. By abuse of notation, we denote $\xi|_{M^{\lambda}}$ by $\xi$.
 According to Theorem \ref{thm4-0-8} and equation (\ref{eq3-48}), we have 
  \begin{equation}\label{eq3-49}
\nabla_{\xi}\tilde{\delta}=\nabla_{\xi}(\imath(\delta)) = \imath( \nabla_{\xi}\delta)= \imath( \nabla_{\xi}(\pi^{*} \eta)).
 \end{equation}
Since $\tilde{\delta} \in \shH_{\mathrm{mix},\lambda}$, $\nabla_{\xi} \tilde{\delta}=0$, for $\xi \in \Gamma(M, \shP_{\mathrm{mix}})$.
By the injectivity of $\imath$ and equation (\ref{eq3-49}), 
we obtain:
\begin{equation} 
 \nabla_{\xi}(\pi^{*} \eta)=0,~~ \forall \xi \in \Gamma(M^{\lambda}, \shP_{\mathrm{mix}}).
\end{equation}
 Then by Theorem \ref{thm4-0-8}, we have $\nabla_{\zeta} \eta=0$, for any $\zeta \in \Gamma(M_{\lambda}, TM_{\lambda}^{0,1})$. This implies $\nabla^{0,1} \eta=0$. By the regularity of elliptic operator $\Delta=\bar{\partial}^{*}\bar{\partial}$, $\eta$ is smooth. Therefore there exists a holomorphic section $s \in H^{0}(M_{\lambda},L_{\lambda})$ such that 
 $\eta= \iota(s)$ under the inclusion map
 $\iota: \Gamma(M_{\lambda}, L_{\lambda}) \rightarrow \Gamma_{c}(M_{\lambda},L_{\lambda}^{-1})'$ with respect to $\vol_{\lambda}$.
 It is remain to show $\tilde{\delta}= \kappa(s)$.
 
  According to the above discussion, we have
  $\tilde{\delta}= \imath (\pi^{*}(\iota(s)))$. Recall that $\kappa(s)=\imath(\delta^{s})$, where 
    $\delta^{s}$ with respect to volume form $\vol^{\lambda}$ is defined by
 \begin{equation} \label{eq3-2-1}
  \delta^{s}(\tau)= \int_{M^{\lambda}} \langle \pi^{*}s, \tau \rangle \vol^{\lambda},
 \end{equation}
  for any test section $\tau \in \Gamma_{c}(M^{\lambda}, L^{\lambda})'$.  
  By the injectivity of $\imath$, to show $\tilde{\delta}= \kappa(s)$, it is enough to show:
  \begin{equation}
  \pi^{*}(\iota(s))=\delta^{s}.
  \end{equation}
 By remark \ref{rk4-1}, we have 
\begin{equation} \label{eq3-2-2}
\int_{M^{\lambda}} \langle \pi^{*}s, \tau \rangle \vol^{\lambda}= (\pi^{*}s)(\tau)=s(\pi_{*} \tau)=\int_{M_{\lambda}} \langle s, \pi_{*}\tau \rangle \vol_{\lambda}
\end{equation}
And 
\begin{equation} \label{eq3-2-3}
\pi^{*}(\iota(s))(\tau)=(\iota(s))(\pi_{*} \tau)=\int_{M_{\lambda}} \langle s, \pi_{*}\tau \rangle \vol_{\lambda}.
\end{equation}
 
According to equations (\ref{eq3-2-1}
), (\ref{eq3-2-2}), and (\ref{eq3-2-3}), we have 
  $\pi^{*}(\iota(s))=\delta^{s}$.
  \end{proof}
\section{Appendix}

\subsection{Polarizations on symplectic manifolds}

A step in the process of geometric quantization is to choose a polarization.
We first recall the definitions polarizations on
symplectic manifolds $(M,\omega )$ (See \cite{SW, Wo}). All polarizations
discussed in this subsection are smooth.

\begin{definition}
\label{def4-1}
A \emph{complex polarization} on $M$ is a complex sub-bundle of the complexified tangent bundle $TM\otimes 
\mathbb{C}$ satisfying the following conditions:
\begin{enumerate}
\item $\mathcal{P}$ is involutive, i.e. if $u,v \in \Gamma(M,\mathcal{P})$,
then $[u,v] \in \Gamma(M,\mathcal{P})$;
\item for every $x \in M$, $\mathcal{P}_{x} \subseteq T_{x}M \otimes {
\mathbb{C}}$ is Lagrangian; and
\item  $\mathrm{rk}_{\RR}\left( \mathcal{P}\right):=\mathrm{rank}(\mathcal{P}\cap \overline{
\mathcal{P}}\cap TM)$ is constant.
\end{enumerate}
Furthermore, $\shP$ is called
\begin{enumerate}
\item  [$\cdot$]{\em real polarization}, if $\mathcal{P}=\overline{\mathcal{P}}$,
i.e. $\mathrm{rk}_{\RR}\left( \mathcal{P}\right) =m$; 
\item  [$\cdot$] {\em K\"ahler polarization}, if $\mathcal{P}\cap \overline{
\mathcal{P}}=0$, i.e. $\mathrm{rk}_{\RR}\left( \mathcal{P}\right) =0$; 
\item [$\cdot$] {\em mixed polarization}, if $0 < \rank(\shP \cap \overline{\shP} \cap TM) < m$, i.e. $0<\mathrm{rk}_{\RR}\left( \mathcal{P}\right) <m$. 
\end{enumerate}
\end{definition}

\subsection{Singular polarizations on symplectic manifolds}

In subsection, we review the definitions of singular polarizations, smooth sections of singular polarizations which were used in the proof of the main results (see \cite{LW1}). 
\begin{definition}\label{def4-3}
 $\shP \subset  TM \otimes\CC$ is a {\em singular complex distribution} on $M$ if it satisfies:
$\shP_{p} $ is a vector subspace of $ T_{p}M \otimes \CC$, for all point $p \in M$.
Such a $\shP$ is called {\em smooth on an open subset $\check{M} \subset M$} if 
 $\shP|_{\check{M}}$ is a smooth sub-bundle of the tangent bundle $T\check{M} \otimes \CC$.
\end{definition}

\begin{remark}
In this paper, we only consider such distributions with mild singularities in the sense that they are only singular outside an open dense subset $\check{M}\subset M$. Under our setting, we define smooth sections of singular distributions and  {\em involutive distributions} as follows.
\end{remark}

\begin{definition}\label{def4-4}
Let $\shP$ be a singular complex distribution of $TM\otimes \CC$. For any open subset $U$ of $M$, {\em the space of smooth sections of $\shP$ on $U$} is defined by the smooth section of $TM\otimes \CC$ with value in $\shP$, that is,
$$ \Gamma(U, \shP) = \{ v \in \Gamma(U, TM \otimes \CC) \mid  v_{p} \in (\shP)_{p}, \forall  p\in U \}.$$
\end{definition}

\begin{definition}\label{def4-5}
Let $\shP$ be a singular complex distribution on $M$. $\shP$ is {\em involutive} if it satisfies:
$$[u, v] \in \Gamma(M,\shP), \mathrm{~for~ any~} u, v \in \Gamma(M, \shP).$$
\end{definition}

\begin{definition}\label{def4-7} Let $\shP$ be a singular complex distribution $\shP$ on $M$ and smooth on $\check{M}$. Such a $\shP$ is called a {\em singular polarization on $M$}, if it satisfies the following conditions:
\begin{enumerate} [label = (\alph*)]
\item $\shP$ is involutive, i.e. if $u, v \in \Gamma(M,\shP)$, then $ [u, v] \in \Gamma(M,\shP)$;
\item for every $x \in \check{M}$, $\shP_{p} \subseteq T_{p}M \otimes \CC$ is Lagrangian; and
\item the real rank $ \mathrm{rk}_{\RR}(\shP):=\rank(\shP \cap \overline{\shP} \cap TM)|_{\check{M}}$ is a constant.
\end{enumerate}
 Furthermore, such a singular $\shP$ is called 
 \begin{enumerate}
 \item   [$\cdot$] {\em real polarization}, if $\shP|_{\check{M}} = \overline{\shP}|_{\check{M}}$, i.e. $\mathrm{rk}_{\RR}(\shP|_{\check{M}})=m$;
 \item [$\cdot$]  {\em K\"ahler polarization}, if $\mathcal{P}_{\check{M}}\cap \overline{
\mathcal{P}}|_{\check{M}}=0$ on $\check{M}$, i.e. $r\left( \shP|_{\check{M}}\right) =0$; 
 \item  [$\cdot$]  {\em mixed polarization}, if $0 < \rank(\shP \cap \overline{\shP} \cap TM)|_{\check{M}} < m$, i.e. $0< \mathrm{rk}_{\RR}(\shP|_{\check{M}})<m$.
 \end{enumerate}
\end{definition}

   \subsection{Coisotropic embedding theorem}
We review the coisotropic embedding theorem studied by Guillemin in \cite{Gui1}, which was used in the proof of taking divergence.
 Let $(M, \omega)$ be a symplectic manifold of dimensional $2m$ equipped with Hamiltonian $T^{n}$-action with moment map $\mu$.
 Without loss of generality, we assume $n=1$. Choose a principal $T^{1}$-connection $\alpha \in \Omega^{1}(M^{0}, \fot)$ on $M^{0}$, where $M^{0}=\mu^{-1}(0)$.
Consider $M^{0}$ as a submanifold of $M^{0} \times \RR$ via the embedding $$i: M^{0} \rightarrow M^{0} \times \RR, ~~i(p)=(p,0).$$ 
On the product space $\tilde{M}=M^{0} \times (-\epsilon, \epsilon)$,
the two-form $$\tilde{\omega}= \pi^{*} \omega_{0} + d(t \alpha), -\epsilon < t < \epsilon$$ is symplectic on $ \tilde{M}$ and satisfies $i^{*} \tilde{\omega}=\pi^{*}\omega_{0}$.
 Extending the $T^{1}$-action on $M^{0}$ to $M^{0} \times \fot^{*}$ in a trivial manner.
Then $\tilde{\omega}$ is $T^{1}$-invariant and that the action of $T^{1}$ on $ M^{0} \times \fot^{*}$ is Hamiltonian with moment map
$$\mu_{0}:  M^{0} \times \fot^{*} \rightarrow \fot^{*}, (p, t) \mapsto t.$$ 

\begin{theorem}\cite[Theorem 2.2]{GS1} In a neighborhood of $M^{0}$, the Hamiltonian $T^{n}$-spaces $(M, \omega)$ and $(\tilde{M}, \tilde{\omega})$ are isomorphic.
\end{theorem}

\subsection{Geometric quantization commute with symplectic reduction} In this subsection, we review the work on geometric quantization commute with symplectic reduction by Guillemin and Sternberg in \cite{GS1}.
Let $(L, \nabla)$ and $(L_{\lambda}, \nabla_{\lambda})$ be the pre-quantum line bundle on $M$ and $M_{\lambda}$ respectively as discussed before, for $
\lambda \in \fot^{*}_{\ZZ, \mathrm{reg}}$.
Then the quantum space $\shH_{P_{J}}$ associated to $\shP_{J}$ is the space of $J$-holomorphic sections of $L$:
$$\shH_{P_{J}}=  \{s \in \Gamma(M, L) \mid \bar{\partial}_{J} s =0\}=H^{0}(M,L).$$

One can perform two processes on the pre-quantum line bundle $(L, \nabla)$; one is geometric quantization, and the other is symplectic reduction.  
  Guillemin and Sternberg in \cite{GS1} showed that these two processes commute with each other, that is,
 \begin{equation}
 (\shH_{P_{J}})_{\lambda} \cong \shH_{P_{J, \lambda}}, 
 \end{equation}
where $ (\shH_{P_{J}})_{\lambda}$ ($J_{\lambda}$-holomorphic sections of $L_{\lambda}$) is the $\lambda$-weight subspace of $ \shH_{P_{J}}$ and $\shH_{P_{J, \lambda}}$ is the quantum space associated to reduced K\"ahler polarization $\shP_{J, \lambda}$, i.e.
\begin{equation}
 \shH_{P_{J, \lambda}}=  \{s \in \Gamma(M_{\lambda}, L_{\lambda}) \mid \bar{\partial}_{J_{\lambda}} s =0\}=H^{0}(M_{\lambda},L_{\lambda}).
\end{equation}

\phantomsection

\bibliographystyle{amsplain}

\begin{thebibliography}{99}

\addcontentsline{toc}{chapter}{Bibliography}  

\bibitem{BFMN}
T. Baier, C. Florentino, J. M. Mour$\tilde{a}$o and J. P. Nunes,
{\em Toric K$\ddot{\rm{a}}$hler metrics seen from infinity, quantization and compact tropical amoebas}, J. Diff. Geom., 89 (3), 411-454,  2011.

\bibitem{Gui1}
V. Guillemin,
{\em Moment Maps and Combinatorial Invariants of Hamiltonian $T^{n}$- spaces}, Progress in Math., 122, Birkh$\ddot{\rm{a}}$user, 1994.


\bibitem{GS1}
V. Guillemin and S. Sternberg,
{\em Geometric Quantization and Multiplicities of Group Representations}, Inventiones mathematicae, 67.3 (1982): 515-538.

\bibitem{GS4}
V. Guillemin and S. Sternberg,
{\em Symplectic Techniques in Physics}, Cambridge University Press, Cambridge University Press, Cambridge, 1984

\bibitem{Ham1}
M. D. Hamilton,
{\em Locally toric manifolds and singular Bohr-Sommerfeld leaves}, Mem. Amer. Math. Soc. 207 (2010), no. 971, vi+60pp. 

\bibitem{HJ}
J. Horv$\acute{a}$th, 
{\em Topological vector spaces and distributions}, Courier Corporation, 2012.

\bibitem{Ki}
A. A. Kirillov,
{\em Geometric quantization}, in: Encyclopaedia of Mathematical Sciences, vol. 4 Dynamical systems, Springer-Verlag, 1990, 137-172.

 \bibitem{Kos}
B. Kostant,
{\em Quantization and unitary representations}, In: Modern analysis and applications. Lecture Notes in Math., Vol. 170, pp. 87-207. Berlin-Heidelberg-Mew York: Springer 1970.

\bibitem{LW1}
 N.C. Leung and D. Wang
 {\em Geodesic rays in space of K\"ahler metrics with T-symmetry}, arXiv preprint arXiv: 2211.05324 (2022).

 \bibitem{MW}
 J. Marsden and A. Weinstein,
 {\em Reduction of symplectic manifolds with symmetry}. Report on Math. Phys. 5,121-130 (1974).

\bibitem{Ru}
W. Rudin,
{\em Functional analysis}, Second edition. International Series in Pure and Applied Mathematics. McGraw-Hill, Inc., New York, 1991. 

\bibitem{SW}
D. Simms and N. Woodhouse,
{\em Lectures on geometric quantization},
Lectures Notes in Physics, Vol. 53. Berlin-Heidelberg-New York: Springer 1976.


\bibitem{We1}
A. Weinstein,
{\em Symplectic manifolds and their Lagrangian submanifolds}, Advances in Math. 6 (1971), 329-346.


\bibitem{Wo}
N. M. J. Woodhouse,
{\em Geometric quantization}, Second Edition, Clarendon Press, Oxford, 1991.


\end{thebibliography}

\end{document}